\documentclass[12pt,a4paper]{amsart}
\usepackage{latexsym,amsfonts,amsmath,amssymb,amsthm,rotating,txfonts,mathrsfs}
\usepackage{epic}
\usepackage{curves}
\input xy
\xyoption{all}

\usepackage{amscd,graphics}
\usepackage{diagrams}

\newcommand{\Proj}{{\rm Proj}}

\newcommand{\Spec}{{\rm Spec}}

\input amssym.def
\input amssym.tex
\newcommand{\nc}{\newcommand}
\nc{\bla}{\phantom{bbbbb}}

\newcommand{\Hom}{ \,{\rm Hom} \,}
\newcommand{\Sym}{ \,{\rm Sym} \,}

\newcommand{\beq}{\begin{equation}}
\newcommand{\eeq}{\end{equation}}
\newcommand{\barr}{\begin{array}}
\newcommand{\earr}{\end{array}}
\newcommand{\beqar}{\begin{eqnarray}}
\newcommand{\eeqar}{\end{eqnarray}}
\newtheorem{theorem}{Theorem}[section]
\newtheorem{corollary}[theorem]{Corollary}
\newtheorem{lemma}[theorem]{Lemma}
\newtheorem{prop}[theorem]{Proposition}
\newtheorem{definition}[theorem]{Definition}
\newtheorem{remark}[theorem]{Remark}
\newtheorem{conjecture}[theorem]{Conjecture}
\newtheorem{exit}[theorem]{Example}

\newenvironment{rem}{\begin{remark}\rm}{\end{remark}}

\newenvironment{defn}{\begin{definition}\rm}{\end{definition}}

\newcommand{\GG}{{\mathbb G }}

\newcommand{\CC}{{\mathbb C }}
\nc{\FF}{ {\mathbb F} }
\nc{\HH}{ {\mathbb H} }
\newcommand{\ZZ}{{\mathbb Z }}
\newcommand{\PP}{ {\mathbb P } }
\newcommand{\QQ}{{\mathbb Q }}
\newcommand{\UU}{{\mathbb U }}

\newcommand{\cale}{\mathcal{E}}

\newcommand{\calo}{\mathcal{O}}

\newcommand{\cals}{\mathcal{S}}

\nc{\conv}{{\rm Conv}}
\nc{\umax}{{U_{\max}}}
\newcommand{\syms}{\mathrm{sym}}

\newcommand{\symdot}{\mathrm{Sym}^{\le k}\CC^n}

\newcommand{\comb}{\mathrm{perm}}
\newcommand{\symk}{\Sym^{\le k}\CC^n}
\newcommand{\symkp}{\syms^{\le k}p}
\newcommand{\symkk}{\Sym^{\le k}\CC^k}
\newcommand{\wsymk}{\wedge^k(\Sym^{\le k}\CC^k)}
\newcommand{\wsymn}{\wedge^n(\Sym^{\le k}\CC^n)}
\newcommand{\grass}{\mathrm{Grass}}

\newcommand{\bg}{\mathbf{\gamma}}
\newcommand{\bX}{\overline{X}}

\newcommand{\be}{\mathbf{e}}
\newcommand{\bz}{\mathbf{z}}

\newcommand{\bi}{\mathbf{i}}

\newcommand{\bu}{\mathbf{u}}

\newcommand{\bl}{\mathbf{\lambda}}
\newcommand{\zdis}{\mathbf{z}}
\newcommand{\SL}{\mathrm{SL}}
\newcommand{\GL}{\mathrm{GL}}

\newcommand{\ba}{\mathbf{\alpha}}

\newcommand{\bs}{\mathbf{s}}
\newcommand{\jetk}[2]{J_{k}({#1},{#2})}
\newcommand{\jetko}[2]{J_{k}^o({#1},{#2})}
\newcommand{\jetreg}[2]{J_{k}^{\mathrm{reg}}({#1},{#2})}
\newcommand{\jetnondeg}[2]{J_{k}^{\mathrm{nondeg}}({#1},{#2})}

\newcommand{\flk}{\mathrm{Flag}_{k}(\CC^n)}
\newcommand{\flag}{\mathrm{Flag}}
\newcommand{\Fl}{\mathcal{F}}
\newcommand{\imm}{\mathrm{im}}

\newcommand{\lieu}{{\mathfrak u}}

\newcommand{\sym}{\mathrm{Sym}}

\newcommand{\symdott}{\mathrm{Sym}^{\le k}\CC^n}

\nc{\lieq}{{\mathfrak q}}
\nc{\liez}{{\mathfrak z}}
\nc{\lieqs}{{\lieq}^*}
\nc{\lieg}{{\mathfrak g}}
\nc{\liegs}{{\lieg}^*}
\nc{\liep}{{\mathfrak p}}
\nc{\lieps}{{\liep}^*}

\def\a{\alpha}
\def\b{\beta}
\def\g{\gamma}
\def\d{\delta}
\def\e{\varepsilon}
\def\vare{\varepsilon}

\def\t{\theta}

\def\l{\lambda}

\def\n{\nu}

\def\s{\sigma}

\def\vp{\varphi}

\def\U{\Upsilon}

\title{A geometric construction for invariant jet differentials}

\author{Gergely Berczi and Frances Kirwan\\Mathematical Institute, Oxford OX1 3BJ, UK}

\thanks{This work was supported by the Engineering and Physical Sciences 
Research Council [grant numbers   GR/T016170/1,EP/G000174/1].}

\begin{document}

\maketitle

\section{Introduction}

The action of the reparametrization group $\GG_k$, consisting of $k$-jets of germs of biholomorphisms of $(\CC,0)$, on the bundle $J_k = J_kT^*X$ of $k$-jets at 0 of germs of holomorphic curves $f:\CC \to X$ in a complex manifold $X$ has been a focus of investigation since the work of Demailly \cite{dem} which built on that of Green and Griffiths \cite{gg}. Here $\GG_k$ is a non-reductive complex algebraic group which is the semi-direct product $\GG_k = \UU_k \rtimes \CC^*$ of its unipotent radical $\UU_k$ with $\CC^*$; it has the form 
$$
\GG_k \cong \left\{ 
\left(\begin{array}{ccccc}
\a_1 & \a_2 & \a_3 & \cdots  & \a_k \\
0        & \a_1^2 &  \cdots &  &  \\
0        & 0       & \a_1^3  & \cdots &  \\
\cdot    & \cdot   & \cdot    & \cdot &  \cdot \\
0 & 0 & 0 & \cdots  & \a_1^k 
\end{array} \right) : \a_1 \in \CC^*, \a_2,\ldots,\a_k \in \CC \right\}
$$
where the entries above the leading diagonal are polynomials in $\a_1, \ldots, \a_k$, and $\UU_k$ is the subgroup consisting of matrices of this form with $\a_1=1$. The bundle of Demailly-Semple jet differentials of order $k$ over $X$ has
fibre at $x \in X$ given by the algebra
$\calo((J_{k})_x)^{\UU_k} $  
of  $\UU_k$-invariant polynomial functions on the fibre $(J_{k})_x = (J_kT^*X)_x$
of $J_kT^*X$. 
More generally following \cite{PR} we can replace $\CC$ with $\CC^p$ for $p\geq 1$ and consider the bundle $J_{k,p}T^*X$ of $k$-jets at 0 of holomorphic maps $f:\CC^p \to X$ and the reparametrization group $\GG_{k,p}$ consisting of $k$-jets of germs of biholomorphisms of $(\CC^p,0)$; then $\GG_{k,p}$ is the semi-direct product of its unipotent radical $\UU_{k,p}$ and the complex reductive group $\GL(p)$, while its subgroup $\GG_{k,p}'= \UU_{k,p} \rtimes \SL(p)$ (which equals $\UU_{k,p}$ when $p=1$) fits into an exact sequence
$1 \to \GG_{k,p}' \to \GG_{k,p} \to \CC^* \to 1$. The generalized Demailly-Semple algebra is then $\calo((J_{k,p})_x)^{\GG_{k,p}'}$.

The Demailly-Semple algebras $\calo(J_{k})^{\UU_k}$ and their generalizations  have been studied for a long time. The invariant jet differentials
play a crucial role in the strategy devised by Green, Griffiths \cite{gg},
Bloch \cite{bloch}, Demailly \cite{dem,demelgoul}, Siu \cite{siu1,siu2,siu3} and others
to prove Kobayashi's 1970 hyperbolicity
conjecture \cite{kob} and the related conjecture of Green and Griffiths in the
special case of hypersurfaces in projective space. This strategy has been recently used successfully by Diverio, Merker and Rousseau in \cite{dmr} and 
then by the first author in \cite{berczi} to give  effective lower bounds for the degrees of generic hypersurfaces in $\PP_n$ for which the Green-Griffiths conjecture holds. 

In particular it has been a long-standing problem to determine whether the algebras of invariants $\calo((J_{k,p})_x)^{\GG_{k,p}'}$ and bi-invariants
$\calo((J_{k,p})_x)^{\GG_{k,p}'\times U_{n,x}}$ (where $U_{n,x}$ is a maximal
unipotent subgroup of $\GL(T_xX) \cong \GL(n)$)
are finitely generated as graded complex
algebras, and if so to provide explicit finite generating sets. In \cite{merker}
Merker showed that when $p=1$ and both $k$ and $n=\dim X$ are small then these algebras are finitely generated, and for $p=1$ and all $k$ and $n$ he provided an
algorithm which produces  finite sets of generators when they exist.  In this paper we will describe methods inspired by \cite{bsz} and
 the approach of \cite{DK} to non-reductive geometric invariant theory (GIT)
to prove the finite generation of $\calo((J_{k})_x)^{\UU_k}$ 
for all $n$ and $k\ge 2$ (from which the finite generation of the corresponding bi-invariants follows).  
In fact we will show that $\UU_k$ is a Grosshans subgroup of
$\SL(k)$, so that 
the algebra $\calo(\SL(k))^{\UU_k}$ is finitely generated and hence
every linear action of $\UU_k$ which extends to a linear action of $\SL(k)$ has finitely generated invariants. We will also give a geometric description of a finite set of generators for $\calo(\SL(k))^{\UU_k}$, and a geometric
description of the associated affine variety
$$\SL(k)/\!/\UU_k = \mathrm{Spec}(\calo(SL(k))^{\UU_k})$$
which leads to a geometric description of the affine variety
$$(J_k)_x/\!/\UU_k = \mathrm{Spec}(\calo((J_k)_x)^{\UU_k})$$
as a GIT quotient 
$$ ((J_k)_x \times (\SL(k)/\!/\UU_k))/\!/\SL(k)$$
by the reductive group $\SL(k)$, in the sense of classical geometric invariant
theory \cite{GIT}.
Similarly we expect that if $p>1$  and $k$ is sufficiently large
(depending on $p$) then $\GG_{k,p}'$ is a subgroup of $\SL(\syms^{\leq k}(p))$, where 
$$\syms^{\leq k}(p) =
\sum_{i=1}^k \dim \sym^i\CC^p,$$ such that the algebra
$\calo(\SL(\syms^{\leq k}(p)))^{\GG_{k,p}'}$ is finitely generated,
and thus that the algebra
and $\calo((J_{k,p})_x)^{\GG_{k,p}'}$
is also finitely generated, and we have a geometric description of the
associated affine variety
$$(J_{k,p})_x/\!/\GG_{k,p}'.$$

The layout of this paper is as follows. $\S$2 reviews the reparametrization groups $\GG_k$ and $\GG_{k,p}$ and their actions on jet bundles and jet differentials over a complex manifold $X$. Next $\S$3 reviews some of the results of \cite{DK} 
on non-reductive geometric invariant theory.
In $\S$4 we recall from \cite{bsz} a geometric description of the quotients by $\UU_k$ and $\GG_k$ of open subsets of $(J_k)_x$, and in $\S$5 this
is used to find explicit affine and projective embeddings of these quotients
and explicit embeddings of $\SL(k)/\UU_k$. In
$\S$6 we see that the complement of $\SL(k)/\UU_k$ in its closure for a suitable embedding in an affine space has codimension at least two. 
 In $\S$7 we conclude that $\UU_k$ is a Grosshans subgroup of $\SL(k)$
when $k \geq 2$, so
 that $\calo(SL(k))^{\UU_k}$ and $\calo((J_{k})_x)^{\UU_k}$ are finitely generated, and 
provide a geometric description of a finite set of generators of  $\calo(\SL(k))^{\UU_k}$.
Finally $\S$8 and $\S$9 discuss how to extend the results of $\S$6 and $\S$7 to the action of $\GG_{k,p}$  on the jet bundle
$J_{k,p} \to X$ of $k$-jets of germs of holomorphic maps from $\CC^p$ to $X$ for $p>1$.

\noindent\textbf{Acknowledgments}
We are indebted to Damiano Testa, who called our attention to the importance of the group $\GG_k$ in the Green-Griffiths problem. We would also like to thank Brent Doran for helpful discussions.
 
The first author warmly thanks Andras Szenes, his former PhD supervisor, for his patience and their joint work from which this paper has grown.

\section{Jets of curves and jet
differentials}\label{jetdifferentials}

Let $X$ be a complex $n$-dimensional manifold and let $k$ be a positive integer. Green and Griffiths
in \cite{gg} introduced the bundle $J_k \to X$
of $k$-jets of germs of parametrized curves in $X$; its
fibre over $x\in X$ is the set of equivalence classes of germs of holomorphic
maps $f:(\CC,0) \to (X,x)$, with the equivalence relation $f\sim g$
if and only if the derivatives $f^{(j)}(0)=g^{(j)}(0)$ are equal for
$0\le j \le k$. If we choose local holomorphic coordinates
$(z_1,\ldots, z_n)$ on an open neighbourhood $\Omega \subset X$
around $x$, the elements of the fibre $J_{k,x}$ are represented by the Taylor expansions 
\[f(t)=x+tf'(0)+\frac{t^2}{2!}f''(0)+\ldots +\frac{t^k}{k!}f^{(k)}(0)+O(t^{k+1}) \]
 up to order $k$ at
$t=0$ of
$\CC^n$-valued
maps
\[f=(f_1,f_2,\ldots, f_n)\]
on open neighbourhoods of 0 in $\CC$. Thus in these coordinates the fibre is
\[J_{k,x}=\left\{(f'(0),\ldots, f^{(k)}(0)/k!)\right\}=(\CC^n)^k,\]
which we identify with $\CC^{nk}$.  Note, however, that $J_k$ is not
a vector bundle over X, since the transition functions are polynomial, but
not linear.

Let $\GG_k$ be the group of $k$-jets  at the origin of local reparametrizations
of $(\CC,0)$
\[t \mapsto \varphi(t)=\a_1t+\a_2t^2+\ldots +\a_kt^k,\ \ \ \a_1\in
\CC^*,\a_2,\ldots,\a_k \in \CC,\] in which the composition law is
taken modulo terms $t^j$ for $j>k$. This group acts fibrewise on
$J_k$ by substitution. A short computation shows that this is a
linear action on the fibre:
\begin{eqnarray*} f \circ
\varphi(t)=f'(0)\cdot(\a_1t+\a_2t^2+\ldots
+\a_kt^k)+\frac{f''(0)}{2!}\cdot (\a_1t+\a_2t^2+\ldots
+\a_kt^k)^2+\ldots \\
\ldots +\frac{f^{(k)}(0)}{k!}\cdot (\a_1t+\a_2t^2+\ldots +\a_kt^k)^k 
\mbox{ (modulo $t^{k+1}$)}
\end{eqnarray*}
so the linear action of $\varphi$ on the $k$-jet $(f'(0),f''(0)/2!,\ldots,
f^{(k)}(0)/k!)$ is given by the following matrix multiplication:
\begin{equation}\label{matrixform}
(f'(0),f''(0)/2!,\ldots,f^{(k)}(0)/k!) \cdot 
\left(\begin{array}{ccccc}
\a_1 & \a_2 & \a_3 & \cdots  & \a_k \\
0        & \a_1^2 & 2\a_1\a_2 & \cdots &  \a_1\a_{k-1}+\ldots +\a_{k-1}\a_1 \\
0        & 0       & \a_1^3  & \cdots & 3\a_1^2\a_{k-2}+\ldots \\
\cdot    & \cdot   & \cdot    & \cdot &  \cdot \\
0 & 0 & 0 & \cdots  & \a_1^k 
\end{array} \right)
\end{equation}
where the matrix has general entry 
\[(G_{k})_{i,j}=\sum_{ s_1 \geq 1,\ldots, s_i \geq 1, \,\,\,
s_1+\ldots +s_i=j }\a_{s_1}\ldots \a_{s_i}\]
for $i,j\le k$. 

There is an exact sequence of groups:
\begin{equation}\label{exactsequence}
1 \rightarrow \UU_k \rightarrow \GG_k \rightarrow \CC^* \rightarrow
1,
\end{equation}
 where $\GG_k \to \CC^*$ is the morphism $\varphi \to
\varphi'(0)=\a_1$ in the notation used above, and
\[ \GG_k=\UU_k \rtimes \CC^*\]
is a semi-direct product. With the above identification, $\CC^*$ is
the subgroup of $\GG_k$ consisting of diagonal matrices satisfying $\a_2=\ldots =\a_k=0$ and
$\UU_k$ is the unipotent radical of $\GG_k$, consisting of matrices of the form above with $\a_1=1$. The action
of  $\l \in \CC^*$ on $k$-jets is thus described by
\[\l\cdot (f'(0),f''(0)/2!,\ldots ,f^{(k)}(0)/k!)=(\l f'(0),\l^2 f''(0)/2!,\ldots,
\l ^kf^{(k)}(0)/k!)\]

Let $\cale_{k,m}^n$ denote the vector space of complex valued polynomial functions
$$Q(u_1,u_2,\ldots, u_k)$$ of $u_1=(u_{1,1},\ldots,u_{1,n}),\ldots,u_k=(u_{k,1},\ldots,u_{k,n})$  of weighted degree $m$ with respect to
this $\CC^*$ action, where $u_i=f^{(i)}(0)/i!$; that is, such that
\[Q(\l u_1,\l^2 u_2,\ldots, \l^k u_k)=\l^m Q(u_1,u_2,\ldots,
u_k).\]
Thus elements of $\cale_{k,m}^n$ have the form
\[Q(u_1,u_2,\ldots, u_k)=\sum_{|i_1|+2|i_2|+\ldots
+k|i_k|=m}u_1^{i_1}u_2^{i_2}\ldots u_k^{i_k},\] where
$i_1=(i_{1,1},\ldots,i_{1,n}),\ldots, i_k=(i_{k,1},\ldots,i_{k,n})$ are multi-indices of length $n$.
There is an induced action of $\GG_k$ on the algebra
$\bigoplus_{m \geq 0} \cale_{k,m}^n$. Following Demailly (see \cite{dem}), we denote by $E_{k,m}^n$ (or $E_{k,m}$) the Demailly-Semple bundle whose fibre at $x$ consists of
the $\UU_k$-invariant polynomials on the fibre of $J_k$ at $x$ of weighted degree $m$, i.e those which satisfy
$$Q((f\circ \varphi)'(0),(f\circ \varphi)''(0)/2!, \ldots, (f\circ
\varphi)^{(k)}(0)/k!)$$
$$=\varphi'(0)^m\cdot Q(f'(0),f''(0)/2!,\ldots, f^{(k)}(0)/k!),$$
and we let $E_k^n=\oplus_m E_{k,m}^n$ denote the Demailly-Semple bundle of graded algebras of
invariants.

We can also consider higher
dimensional holomorphic surfaces in $X$, and therefore we fix a
parameter $1\le p \le n$, and study germs of maps $\CC^p \to X$.

Again we fix the degree $k$ of our map, and introduce the bundle
$J_{k,p} \to X$ of $k$-jets of maps $\CC^p \to X$. The
fibre over $x\in X$ is the set of equivalence classes of germs of holomorphic
maps $f:(\CC^p,0) \to (X,x)$, with the equivalence relation $f\sim
g$ if and only if all derivatives $f^{(j)}(0)=g^{(j)}(0)$ are equal for
$0\le j \le k$.

We need a description of the fibre $J_{k,p,x}$ in terms of local
coordinates as in the case when $p=1$.  Let $(z_1,\ldots, z_n)$ be local holomorphic
coordinates  on an open neighbourhood $\Omega \subset X$ around $x$,
and let $(u_1,\ldots, u_p)$ be local coordinates on $\CC^p$. The elements
of the fibre $J_{k,p,x}$ are $\CC^n$-valued maps
\[f=(f_1,f_2,\ldots, f_n)\]
on $\CC^p$,
and two maps represent the same jet if their Taylor expansions around
$\bz=0$
\[f(\bz)=x+\bz f'(0)+\frac{\bz^2}{2!}f''(0)+\ldots +\frac{\bz^k}{k!}f^{(k)}(0)+O(\bz^{k+1}) \]
coincide up to order $k$. Note that here
\[f^{(i)}(0)\in \Hom(\Sym^i\CC^p,\CC^n)\]
and in these coordinates the fibre is a finite-dimensional vector space
\[J_{k,p,x}=\left\{(f'(0),\ldots, f^{(k)}(0)/k!)\right\} \cong \CC^{n {k+p-1 \choose k-1}}.\]

Let $\GG_{k,p}$ be the group of $k$-jets of germs of biholomorphisms
of $(\CC^p,0)$. Elements of $\GG_{k,p}$ are represented by holomorphic maps
\begin{equation} \bu \to \varphi(\bu)=\Phi_1\bu+\Phi_2\bu^2+\ldots +\Phi_k\bu^k=
\sum_{\mathbf{i}\in \ZZ^p \setminus {0}}a_{i_1\ldots
i_p}u_1^{i_1}\ldots u_p^{i_p},\ \ \Phi_1 \text{ is non-degenerate}
\end{equation}
where $\Phi_i \in \Hom(\Sym^i\CC^p,\CC^p)$.
The group $\GG_{k,p}$ admits a natural fibrewise right action on $J_{k,p}$,
by reparametrizing the $k$-jets of holomorphic
$p$-discs. A computation similar to that in \cite{bsz} shows that

\[f \circ\varphi(\bu)=f'(0)\Phi_1\bu+(f'(0)\Phi_2+\frac{f''(0)}{2!}\Phi_1^2)\bu^2 
+ \ldots +\sum_{i_1+\ldots +i_l=d}
\frac{f^{(l)}(0)}{l!}\Phi_{i_1}\ldots \Phi_{i_l}\bu^l.\]
This defines a linear action of $\GG_{k,p}$ on the fibres $J_{k,p,x}$ 
of $J_{k,p}$ with the matrix representation
given by

\begin{equation}
\left(
\begin{array}{ccccc}
\Phi_1 & \Phi_2 & \Phi_3 & \ldots & \Phi_k \\
0 & \Phi_1^2 & \Phi_1\Phi_2 & \ldots &       \\
0 & 0 & \Phi_1^3 & \ldots & \\
. & . & . & . & . \\
& & & & \Phi_1^k
\end{array}
\right),
\end{equation}

where
\begin{itemize}
\item $\Phi_i \in \Hom(\Sym^i\CC^p,\CC^p)$ is a $p \times \dim (\sym^i \CC^p)$-matrix, the $i$th
degree component of the map $\Phi$, which is represented by a
map $(\CC^p)^{\otimes i} \to \CC^p$;
\item $\Phi_{i_1} \ldots \Phi_{i_l}$ is the matrix of the map $\sym^{i_1+\ldots
+i_l}(\CC^p)\to \sym^{l}\CC^p$, which is represented by
\[\sum_{\s \in \mathcal{S}_l}\Phi_{i_1} \otimes \cdots \otimes
\Phi_{i_l}:(\CC^p)^{\otimes i_1} \otimes \cdots \otimes
(\CC^p)^{\otimes i_l} \to (\CC^p)^{\otimes l};\]
\item the $(l,m)$ block of $\GG_{k,p}$ is $\sum_{i_1+\ldots +i_l=m}\phi_{i_1}\ldots \Phi_{i_l}$. The entries in these boxes are indexed by pairs $(\tau,\mu)$ where $\tau \in {p+l-1 \choose l-1}, \mu \in {p+m-1 \choose m-1}$ correspond to bases of $\Sym^l(\CC^p)$ and $\Sym^m(\CC^p)$.
\end{itemize}

\begin{exit}
For $p=2,k=3$, using the standard basis 
\[\left\{e_i,
e_ie_j,e_ie_je_k: 1\le i\le j\le k\le 2\right\}\]
of $(J_{3,2})_x$, we get the following
$9\times 9$ matrix for a general element of $\GG_{3,2}$:

\begin{equation}
\tiny{ \left(
\begin{array}{ccccccccc}
\a_{10} & \a_{01} & \a_{20} & \a_{11} & \a_{02} & \a_{30} & \a_{21} &
\a_{12} & \a_{03} \\
\b_{10} & \b_{01} & \b_{20} & \b_{11} & \b_{02} & \b_{30} & \b_{21} &
\b_{12} & \b_{03} \\
0 & 0 & \a_{10}^2 & \a_{10}\a_{01} & \a_{01}^2 & \a_{10}\a_{20} & \a_{10}\a_{11}+\a_{01}\a_{20} & \a_{10}\a_{02}+\a_{11}\a_{01} & \a_{01}\a_{02}\\
0 & 0 & \a_{10}\b_{10} & \a_{10}\b_{01}+\a_{01}\b_{10} & \a_{01}\b_{01} & \a_{10}\b_{20}+\a_{20}\b_{10} & P & Q & \a_{01}\b_{02}+\a_{02}\b_{01} \\
0 & 0 & \b_{10}^2 & \b_{10}\b_{01} & \b_{01}^2 & \b_{10}\b_{20} & \b_{10}\b_{11}+\b_{20}\b_{01} & \b_{01}\b_{11}+\b_{02}\b_{10} & \b_{01}\b_{02} \\
0 & 0 & 0 & 0 & 0 & \a_{10}^3 & \a_{10}^2\a_{01} & \a_{10}\a_{01}^2 &
\a_{01}^3 \\
0 & 0 & 0 & 0 & 0 & \a_{10}^2\b_{10} & \a_{10}\a_{10}\b_{01} &
\a_{10}\a_{01}\b_{01} & \a_{01}\b_{01}^2 \\
0 & 0 & 0 & 0 & 0 & \a_{10}\b_{10}^2 & \a_{10}\b_{10}\b_{01} &
\a_{10}\b_{01}\b_{01} & \a_{01}\b_{01}^2 \\
0 & 0 & 0 & 0 & 0 & \b_{10}^3 & \b_{10}^2\b_{01} & \b_{10}\b_{01}^2 &
\b_{01}^3 \\
\end{array}
\right)}
\end{equation}
where 
\[P=\a_{10}\b_{11}+\a_{11}\b_{10}+\a_{20}\b_{01}+\a_{01}\b_{20}
\text{ and } Q=\a_{01}\b_{11}+\a_{11}\b_{01}+\a_{02}\b_{10}+\a_{10}\b_{02}.\]
This is a subgroup of the standard parabolic $P_{2,3,4} \subset GL(9)$. The
diagonal blocks are the representations $\sym^i \CC^2$ for
$i=1,2,3$ of $GL(2)$, where $\CC^2$ is the standard representation of $GL(2)$.
\end{exit}

In general the linear group $\GG_{k,p}$ is generated along its first $p$ rows;
that is, the parameters in the first $p$ rows are independent, and
all the remaining entries are polynomials in these parameters. The
assumption on the parameters is that the determinant of the smallest
diagonal $p\times p$ block is nonzero; for the $p=2,k=3$ example
above this means that
\[ \det \left(\begin{array}{cc} \a_{10} & \a_{01} \\
\b_{10} & \b_{01} \end{array} \right) \neq 0.\]

The parameters in the $(1,m)$ block are indexed by a basis of $\Sym^m(\CC^p) \times \CC^p$, so they are of the form $\a_{\nu}^l$ where $\nu \in {p+m-1 \choose m-1}$ is an $m$-tuple and $1\le l \le p$. An easy computation shows that:

\begin{prop}\label{entry} The polynomial in the $(l,m)$ block and entry indexed by $$\tau=(\tau[1],\ldots, \tau[l]) \in {p+l-1 \choose l-1}$$ and $\nu \in {p+m-1 \choose m-1}$ is 
\begin{equation}
(\GG_{k,p})_{\tau,\nu}=\sum_{\nu_1+\ldots+\nu_l=\nu}\a_{\nu_1}^{\tau[1]}\a_{\nu_2}^{\tau[2]} \ldots \a_{\nu_l}^{\tau[l]}
\end{equation}
\end{prop} 

Note that $\GG_{k,p}$ is an extension of its unipotent radical $\UU_{k,p}$ by
$GL(p)$; that is, we have an exact sequence
\[1 \rightarrow \UU_{k,p} \rightarrow \GG_{k,p} \rightarrow GL(p) \rightarrow
1,\] and $\GG_{k,p}$ is the semi-direct product $\UU_{k,p}\rtimes GL(p)$. Here $\GG_{k,p}$ has dimension $p \times \syms^{\leq k}(p)$ where
$\syms^{\leq k}(p) = \dim (\oplus_{i=1}^k \sym^i \CC^p)$, and is a subgroup of the standard 
parabolic subgroup $P_{p,\syms^2(p),\ldots ,\syms^k(p)}$
of $GL(\syms^{\leq k}(p))$
where $\syms^i(p) = \dim(\sym^i\CC^p)$. We define $\GG_{k,p}'$ to be the subgroup
of $\GG_{k,p}$ which is  the semi-direct product 
$$\GG_{k,p}' = \UU_{k,p}\rtimes SL(p)$$
(so that $\GG_{k,p}' = \UU_{k,p}$ when $p=1$) fitting into the exact 
sequence
\[1 \rightarrow \UU_{k,p} \rightarrow \GG_{k,p}' \rightarrow SL(p) \rightarrow
1.\]

The action of the maximal torus $(\CC^{*})^p\subset GL(p)$ of the
Levi subgroup of $\GG_{k,p}$ is
\begin{equation}\label{nweights}
(\l_1,\ldots, \l_p)\cdot f^{(i)}=(\l_1^i\frac{\partial^if}{\partial
u_1^i},\ldots , \l_1^{i_1}\cdots \l_p^{i_p}
\frac{\partial^if}{\partial u_1^{i_1} \cdots \partial u_p^{i_p}}
\ldots \l_p^{i} \frac{\partial^if}{\partial u_p^i} )
\end{equation}

We introduce the {\em Green-Griffiths} vector bundle
$E_{k,p,m}^{GG} \to X$, whose fibres are complex-valued polynomials
$$Q(f'(0),f''(0)/2!,\ldots ,f^{(k)}(0)/k!)$$ 
on the fibres of $J_{k,p}$, having
weighted degree $(m, \ldots, m)$ with respect to the action
(\ref{nweights}) of $(\CC^*)^p$. That is, for $Q\in E_{k,p,m}^{GG}$
\[Q(\bl f'(0),\bl f''(0)/2! ,\ldots, \bl f^{(k)}(0)/k!)=\l_1^m \cdots \l_p^mQ(f'(0),f''(0)/2!, \ldots ,f^{(k)}(0)/k!)\]
for all $\bl \in \CC^p$ and $(f'(0),f''(0)/2!,\ldots, f^{(k)}(0)/k!) \in
J_{k,p,m}$.

\begin{definition}
The generalized Demailly-Semple bundle $E_{k,p,m}\to X$
over $X$ has fibre consisting of the $\GG_{k,p}'$-invariant jet differentials of
order $k$ and weighted degree $(m,\ldots, m)$; that is, the
 complex-valued polynomials
$Q(f'(0),f''(0)/2!,\ldots ,f^{(k)}(0)/k!)$ on the fibres of $J_{k,p}$ which 
transform under any reparametrization $\phi\in \GG_{k,p}$ of
$(\CC^p,0)$ as
\[Q(f \circ \phi)=(J_{\phi})^mQ(f)\circ \phi,\]
where $J_\phi = \det \Phi_1$ denotes the Jacobian of $\phi$ at 0. The
generalized Demailly-Semple bundle of algebras $E_{k,p}=\oplus_{m\geq 0}
E_{k,p,m}$ is the associated graded algebra of $\GG_{k,p}'$-invariants, whose
fibre at $x \in X$ is the generalized Demailly-Semple algebra $\calo((J_{k,p})_x)^{\GG_{k,p}'}$.
\end{definition}

The determination of a suitable generating set for the invariant jet differentials when $p=1$ is important in the longstanding 
 strategy to prove the Green-Griffiths conjecture. It has been suggested in a series of papers \cite{gg,dem,rousseau,merker,dmr, merker2}
that the Schur decomposition of the Demailly-Semple algebra, together with good
estimates of the higher Betti numbers of the Schur bundles and an asymptotic estimation of the Euler charactristic, should result in a positive lower bound for the global sections of the Demailly-Semple jet differential bundle. 

\section{Geometric invariant theory}

Suppose now that $Y$ is a complex quasi-projective variety on which a 
linear algebraic group $G$ acts.
 For geometric invariant theory (GIT)
we need a  {linearization} of the action; that is, a
line bundle $L$ on $Y$ and a lift $\mathcal{L}$ of the action of $G$ to $L$.
 Usually $L$ is ample, and hence (as it makes no difference for GIT
if we replace $L$ with $L^{\otimes k}$ for any integer $k>0$)
 we can assume that for some
projective embedding
$Y \subseteq \PP^n$
the action
of $G$ on $Y$ extends to an action on $\PP^n$ given by a
representation
$\rho:G\rightarrow GL(n+1),$
and take for $L$ the hyperplane line bundle on $\PP^n$.

For classical GIT developed by Mumford \cite{GIT} (cf. also \cite{Dolg,Mukai,New,PopVin}) we require the complex algebraic group $G$ to be reductive. 
Let $Y$ be a {projective} complex variety with an action of a
complex reductive  
group $G$ and linearization $\mathcal{L}$ with respect to an ample line
bundle $L$ on $Y$. Then $y \in Y$ is {\em
semistable} for this linear action if there exists some $m > 0$
and $f \in H^0(Y, L^{\otimes m})^G$ not vanishing at $y$,
and $y$ is {\em stable} if also the action of $G$
on 
the open subset
$$ Y_f := \{ x \in Y \ | \ f(x) \neq 0 \}$$
 is closed with all stabilizers finite.
 $Y^{ss}$ has
a projective categorical quotient $Y^{ss} \to Y/\!/G$, which restricts 
on the set of stable points to a geometric
quotient $Y^s \to Y^s/G$ (see \cite{GIT} Theorem 1.10). The morphism 
$Y^{ss} \to Y/\!/G$ is surjective, and identifies $x, y \in Y^{ss}$ if and only if the closures of the $G$-orbits of $x$ and $y$ meet in $Y^{ss}$; moreover each point in $Y/\!/G$ is represented by a unique closed $G$-orbit in $Y^{ss}$.
There is an induced action of $G$ on the
homogeneous coordinate ring
\[\hat{\calo}_L(Y) = \bigoplus_{k \geq 0} H^0(Y, L^{\otimes k}) \]
of $Y$.
The
subring ${\hat{\calo}}_L(Y)^G$ consisting of the elements of ${\hat{\calo}}_L(Y)$
left invariant by $G$ is a finitely generated graded complex algebra
because $G$ is reductive, and
 the GIT quotient $Y/\!/G$ is the projective variety  $\Proj({\hat{\calo}}_L(Y)^G)$ \cite{GIT}.
The subsets $Y^{ss}$ and $Y^s$ of $Y$ are characterized by the following properties (see
\cite[Chapter 2]{GIT} or \cite{New}).

\begin{prop} (Hilbert-Mumford criteria)
\label{sss} (i) A point $x \in Y$ is semistable (respectively
stable) for the action of $G$ on $Y$ if and only if for every
$g\in G$ the point $gx$ is semistable (respectively
stable) for the action of a fixed maximal torus of $G$.

\noindent (ii) A point $x \in Y$ with homogeneous coordinates $[x_0:\ldots:x_n]$
in some coordinate system on $\PP^n$
is semistable (respectively stable) for the action of a maximal
torus of $G$ acting diagonally on $\PP^n$ with
weights $\a_0, \ldots, \a_n$ if and only if the convex hull
$$\conv \{\a_i :x_i \neq 0\}$$
contains $0$ (respectively contains $0$ in its interior).
\end{prop}

Similarly if a complex reductive group $G$ acts linearly on an affine variety
$Y$ then we have a GIT quotient 
$$Y/\!/G = \Spec (\calo(Y)^G)$$
which is the affine variety associated to the finitely generated algebra
$\calo(Y)^G$ of $G$-invariant regular functions on $Y$. In this case $Y^{ss} = Y$ and the inclusion $\calo(Y)^G \hookrightarrow \calo(Y)$ induces a morphism of
affine varieties $Y \to Y/\!/G$. 

Now suppose that $H$ is any complex linear algebraic
group, with unipotent radical $U \unlhd H$ (so that
$R=H/U$ is reductive and $H$ is isomorphic to the semi-direct product $ U \rtimes R$), acting linearly on a complex projective variety $Y$ with respect to an ample line bundle $L$. Then  $\Proj({\hat{\calo}}_L(Y)^H)$ is not in general well-defined as a projective variety,
since the ring of invariants
$${\hat{\calo}}_L(Y)^H = \bigoplus_{k \geq 0} H^0(Y, L^{\otimes k})^H$$
is not necessarily finitely generated as a graded complex algebra, and so it is not obvious how GIT might be generalised to this situation (cf. \cite{DK,F1,F2,GP1,GP2,KPEN}). However in some cases it is known that $\hat{\calo}_L(Y)^U$ is finitely generated, which implies that 
\[{\hat{\calo}}_L(Y)^H = \left(\bigoplus_{k \geq 0} H^0(Y, L^{\otimes k})^U \right)^{H/U}\] 
is finitely generated and hence the {\em enveloping quotient} in the sense of \cite{DK}is given by the associated projective variety
\[Y/\!/H=\Proj(\hat{\calo}_L(Y)^H).\] 
Similarly if $Y$ is affine and $H$ acts linearly on $Y$ with $\calo(Y)^H$ finitely generated, then we have the enveloping quotient 
$$Y/\!/H = \Spec (\calo(Y)^H).$$

There is a morphism 
\[q: Y^{ss} \to Y/\!/H,\] 
from an open subset $Y^{ss}$ of $Y$ (where $Y^{ss} = Y$ when $Y$ is affine), which restricts to a geometric quotient 
\[q:Y^s \to Y^s/H\]
for an open subset $Y^s \subset Y^{ss}$. However in contrast with the reductive
case,  the morphism $q: Y^{ss} \to Y/\!/H$ is not in general surjective; indeed the image of $q$ is not in general a subvariety of $Y/\!/H$, but is only a constructible subset.

If there is a complex reductive group $G$ containing the unipotent radical $U$
of $H$ such that the algebra $\calo(G)^U$ is finitely generated and the action
of $U$ on $Y$ extends to a linear action of $G$, then 
$$\calo(Y)^U \cong (\calo(Y) \otimes \calo(G)^U)^G$$
is finitely generated and hence so is
$$\calo(Y)^H = (\calo(Y)^U)^{H/U}$$
(or if $Y$ is projective with an ample linearisation $L$ then
$\hat{\calo}_L(Y)^U$ is finitely generated and hence so is
$\hat{\calo}_L(Y)^H$). In this situation we say that $U$ is a Grosshans 
subgroup of $G$ (cf. \cite{Grosshans,Grosshans2}). Then geometrically
$G/U$ is a quasi-affine variety with $\calo(G/U) \cong \calo(G)^U$, and 
it has a
canonical affine embedding as an open subvariety of the affine variety
$$G/\!/U = \mathrm{Spec}(\calo(G)^U)$$
with complement of codimension at least two. Moreover if a linear action
of $U$ on an affine variety $Y$ extends to a linear action of $G$ then
$$Y/\!/U \cong (Y \times G/\!/U)/\!/G$$
(and a corresponding result is true if $Y$ is projective). Conversely if
we can find an embedding of $G/U$ as an open subvariety of an affine variety
$Z$ with complement of codimension at least two, then
$$\calo(G)^U \cong \calo(Z)$$
is finitely generated and $G/\!/U \cong Z$.

Suppose that  $U$ is a unipotent group with a reductive group $R$
of automorphisms of $U$ given by a homomorphism $\phi:R \to \mathrm{Aut}(U)$
such that $R$ contains a central one-parameter subgroup
$\lambda:\CC^* \to R$ for which the weights of
the induced $\CC^*$ action on the Lie algebra $\lieu$ of
$U$ are all nonzero.  Then we can form the semi-direct
product
$$\hat{U} = \CC^* \ltimes U \subseteq R \ltimes U$$
given by $\CC^* \times U$ with group multiplication
$$(z_1,u_1).(z_2,u_2) = (z_1 z_2, (\lambda(z_2^{-1})(u_1))u_2).$$
The groups $\GG_k = \UU_k \rtimes \CC^*$ and $\GG_{k,p} = \UU_{k,p}
\rtimes \GL(p)$ which act on the fibres of the jet bundles
$J_k$ and $J_{k,p}$ are of this form. We will use this structure to study the Demailly-Semple
algebras of invariant jet differentials
$E_k^n$ and $E_{k,p}^n$ and prove

\begin{theorem}\label{fingenerated}
The fibres $\calo((J_{k})_x)^{\UU_{k}} $ 
and $\calo((J_{k,p})_x)^{\GG_{k,p}'} $ 
of the bundles $E_k^n$ and $E_{k,p}^n$ are finitely generated
graded complex algebras. 
\end{theorem}

Thus we have  non-reductive GIT quotients 
$$ (J_{k})_x /\!/\UU_{k} = \Spec(\calo((J_{k})_x)^{\UU_{k}} )$$
and
$$ (J_{k,p})_x /\!/\GG'_{k,p} = \Spec(\calo((J_{k,p})_x)^{\GG_{k,p}'})$$ and we would like to understand them geometrically. There is a crucial difference here from the case of reductive group actions, even though the invariants are finitely generated: when $H$ is a non-reductive group we cannot describe $Y/\!/H$ geometrically as $Y^{ss}$ modulo some equivalence relation. Instead
our aim is to
 use methods inspired by \cite{bsz} to study these 
geometric invariant theoretic quotients and the associated algebras of invariants.

Here a crucial ingredient would be to find an open subset $W$ of $(J_{k,p})_x $ with a geometric quotient $W/\GG_{k,p}'$ embedded as an open
subset of an affine variety $Z$ such that the complement of $W/\GG_{k,p}'$ in
$Z$ has (complex) codimension at least two, and the complement of $W$ in
$(J_{k,p})_x $ has codimension at least two. For then we would have
$$\calo((J_{k,p})_x ) = \calo(W)$$
and
$$\calo((J_{k,p})_x )^{\GG_{k,p}'} = \calo(W)^{\GG_{k,p}'}
= \calo(W/\GG_{k,p}') = \calo(Z),$$
and it follows that $\calo((J_{k,p})_x )^{\GG_{k,p}'}$ is finitely generated since $Z$ is affine, and that
$$Z = \Spec (\calo(Z)) = \Spec(\calo((J_{k,p})_x)^{\GG_{k,p}'})
= ((J_{k,p})_x ) /\!/\GG_{k,p}'.$$
Similarly if we can find a complex reductive group $G$ containing $\GG_{k,p}'$
as a subgroup, and an embedding of $G/\GG_{k,p}'$ as an open subset of an affine
variety $Z$ with complement of codimension at least two, then $\calo(G)^{\GG_{k,p}'}$ is finitely generated. 
It follows as above that if $Y$ is any affine variety on which $G$
acts linearly then
$$\calo(Y)^{\GG_{k,p}'} \cong (\calo(Y) \otimes \calo(G)^{\GG_{k,p}'})^G$$
is finitely generated, and hence so is $ \calo(Y)^{\GG_{k,p}} =
(\calo(Y)^{\GG_{k,p}'})^{\CC^*}$, and similarly 
$\hat{\calo}_L(Y)^{\GG_{k,p}'} $ and $ \hat{\calo}_L(Y)^{\GG_{k,p}}$
are finitely generated if $Y$ is any projective variety wtih an ample line bundle $L$ on which $G$ acts linearly.

We can use the ideas of \cite{bsz} to look for suitable affine varieties $Z$ as above, and in particular to prove

\begin{theorem} \label{Gros} 
$\GG_{k,p}'$ is a 
 subgroup of the special linear group 
$ \SL(\mathrm{sym}^{\le k}p)$ where 
$$ \mathrm{sym}^{\le k}p = \sum_{i=1}^k \dim \sym^i \CC^p =  \left( \begin{array}{c} k+p-1\\ k-1 \end{array} \right) $$ 
such that the algebra of invariants $\calo(\SL(\mathrm{sym}^{\le k}p))^{\GG_{k,p}'}$ is finitely generated, and
every linear action of $\GG_{k,p}'$ or $\GG_{k,p}$ on an affine
or projective variety (with an ample linearisation) which extends to a linear action of $ \GL(\mathrm{sym}^{\le k}p)$ has finitely generated invariants.
\end{theorem}

Theorem \ref{fingenerated} is an immediate consequence of this theorem, since
the action of $\GG_{k,p}$ on  $(J_{k,p})_x$ extends to an action of the general linear group
$ \GL(\mathrm{sym}^{\le k}p)$.
Moreover we will find a geometric description of 
$$ \SL(\mathrm{sym}^{\le k}p)/\!/\GG_{k,p}' \cong
\mathrm{Spec}(\calo( \SL(\mathrm{sym}^{\le k}p))^{\GG_{k,p}'})
$$
and thus a geometric description of
$$ (J_{k,p})_x/\!/\GG_{k,p}' \cong
((J_{k,p})_x \times 
\SL(\mathrm{sym}^{\le k}p)/\!/\GG_{k,p}') /\!/ \SL(\mathrm{sym}^{\le k}p).
$$

\section{A description via test curves}\label{quotientgrass}

In \cite{bsz}  
the action of $\GG_k$ on jet bundles is studied using an idea coming from global singularity
theory. The construction goes as follows.

If $u,v$ are positive integers, let $J_k(u,v)$ denote the vector
space of $k$-jets of holomorphic maps $(\CC^u,0) \to (\CC^v,0)$ at
the origin; that is, the set of equivalence classes of maps
$f:(\CC^u,0) \to (\CC^v,0)$, where $f\sim g$ if and only if
$f^{(j)}(0)=g^{(j)}(0)$ for all $j=1,\ldots ,k$.

With this notation, the fibres of $J_k$ are isomorphic to
$J_k(1,n)$, and the group $\GG_k$ is simply $\jetk 11$ with the
composition action on itself.

If we fix local coordinates $z_1,\ldots, z_u$ at $0\in \CC^u$ we can
again identify the $k$-jet of $f$, using derivatives at the
origin, with $(f'(0),f''(0)/2!,\ldots, f^{(k)}(0)/k!)$, where
$f^{(j)}(0)\in \mathrm{Hom}(\mathrm{Sym}^j\CC^u,\CC^v)$. This way we
get an identification
\[J_k(u,v)=\oplus_{j=1}^k\mathrm{Hom}(\mathrm{Sym}^j\CC^u,\CC^v).\]
We can compose map-jets via substitution and elimination of terms
of degree greater than $k$; this leads to the composition maps
\begin{equation}
  \label{comp}
\jetk vw \times \jetk uv \to \jetk uw,\;\;  (\Psi_2,\Psi_1)\mapsto
\Psi_2\circ\Psi_1 \mbox{modulo terms of degree $>k$ }.
\end{equation}
When $k=1$, $J_1(u,v)$ may be identified with $u$-by-$v$ matrices,
and \eqref{comp} reduces to multiplication of matrices.

The $k$-jet of a curve $(\CC,0) \to (\CC^n,0)$ is simply an element of
$J_k(1,n)$. We call such a curve $\varphi$ {\em regular} if
$\varphi'(0)\neq 0$. Let us introduce the notation $\jetreg 1n$ for the set
of regular curves:
\[\jetreg 1n=\left\{\g \in \jetk 1n; \g'(0)\neq 0 \right\}.\]
Note that if $n>1$ then the complement of $\jetreg 1n$ in $\jetk 1n$ has codimension at least two.
Let $N \ge n$ be any integer and define
\[\U_k=\left\{\Psi\in J_k(n,N):\exists \g \in \jetreg 1n: \Psi \circ \g=0
\right\}\]
to be  the set of those $k$-jets which take at
least one regular curve to zero. By definition, $\U_k$ is the image
of the closed subvariety of $\jetk nN \times \jetreg 1n$ defined by
the algebraic equations $\Psi \circ \g=0$, under the projection to
the first factor. If $\Psi \circ \gamma=0$, we call $\g$ a {\em test
curve} of $\Psi$. 

This term originally comes from global singularity
theory, where this is called the test curve model of $A_k$-singularities. In global singularity theory singularities of polynomial maps $f:(\CC^n,0) \to (\CC^m,0)$ are classified by their local algebras, and  
\[\Sigma_k=\{f \in \jetk nm:\CC[x_1,\ldots, x_n]/\langle f_1,\ldots, f_m \rangle \simeq \CC[t]/t^{k+1}\}\]
is called a Morin singularity, or $A_k$-singularity. The test curve model of Gaffney \cite{gaffney} tells us that
\[\overline{\Sigma_k}=\overline{\U_k}\]
in $\jetk nm$. 

A basic but crucial observation is the following. If $\g$ is a test
curve of $\Psi \in \U_k$, and $\vp \in \jetreg 11=G_k$ is a
holomorphic reparametrization of $\CC$, then $\g \circ \vp$ is,
again, a test curve of $\Psi$:
\begin{diagram}[LaTeXeqno,labelstyle=\textstyle]
\label{basicidea}
  \CC & \rTo^\vp & \CC & \rTo^\g & \CC^n & \rTo^{\Psi} & \CC^N
\end{diagram}
\[\Psi \circ \g=0\ \ \Rightarrow \ \ \ \Psi \circ (\g \circ \vp)=0.\]
In fact, we get all test curves of $\Psi$ in this way from a single
$\gamma$ if the
following open dense property holds: the linear part of $\Psi$ has
$1$-dimensional kernel. Before stating this more precisely in Proposition
\ref{modelprops} below, let us write down the equation $\Psi \circ
\g=0$ in coordinates in an illustrative case. Let
$\g=(\g',\g'',\ldots, \g^{(k)})\in \jetreg 1n$ and
$\Psi=(\Psi',\Psi'',\ldots, \Psi^{(k)})\in \jetk nN$ be the
$k$-jets. Using the chain rule, the equation $\Psi \circ \g=0$ reads
as follows for $k=4$:
\begin{eqnarray} \label{eqn4}
& \Psi'(\g')=0, \\ \nonumber & \frac{1}{2!}\Psi'(\g'')+\Psi''(\g',\g')=0, \\
\nonumber
& \frac{1}{3!}\Psi'(\g''')+\frac{2}{2!}\Psi''(\g',\g'')+\Psi'''(\g',\g',\g')=0, \\
&
\frac{1}{4!}\Psi'(\g'''')+\frac{2}{3!}\Psi''(\g',\g''')+\frac{1}{2!2!}\Psi''(\g'',\g'')+
\frac{3}{2!}\Psi'''(\g',\g',\g'')+\Psi''''(\g',\g',\g',\g')=0.
\nonumber
\end{eqnarray}

\begin{defn} \label{4.1minus}
To simplify our formulas we introduce the following notation for a
partition $\tau=[i_1\ldots i_l]$ of the integer $i_1+\ldots +i_l$:
\begin{itemize}
\item
 the
\textsl{length}: $|\tau|=l$,
\item  the \textsl{sum}: $\sum\tau=i_1+\ldots +i_l$,
\item the \textsl{number of permutations}: $\comb(\tau)$ is
  the number of different sequences consisting of the numbers
  $i_1,\dots, i_l$ (e.g. $\comb([1,1,1,3])=4$),
\item $\bg_\tau = \prod_{j=1}^l \g^{(i_j)}\in\Sym^l\CC^n\;\text{ and }\;
\Psi(\bg_\tau)=\Psi^l(\g^{(i_1)},\dots,\g^{(i_l)})\in\CC^N$.
\end{itemize}
\end{defn}

\begin{lemma}\label{explgp} Let
$\g=(\g',\g'',\ldots, \g^{(k)})\in \jetreg 1n$ and
$\Psi=(\Psi',\Psi'',\ldots, \Psi^{(k)})\in \jetk nN$ be $k$-jets.
Then the equation $\Psi\circ \g = 0$ is equivalent to
  the following system of $k$ linear equations with values in
  $\CC^N$:
\begin{equation}
  \label{modeleq}
\sum_{\tau\in\Pi[m]} \frac{\comb(\tau)}{\prod_{i\in \tau}i!}
\,\Psi(\bg_\tau)=0,\quad m=1,2,\dots, k,
\end{equation}
where $\Pi[m]$ denotes the set of all partitions of $m$.
\end{lemma}

For a given $\g \in \jetreg 1n$ let $\cals_\g$ denote the set of
solutions of \eqref{modeleq}; that is,
\[\cals_\g=\left\{\Psi \in \jetk nN;\Psi \circ \g=0 \right\}.\]
The equations \eqref{modeleq} are linear in $\Psi$, hence
\[\cals_\g \subset \jetk nN\]
is a linear subspace of codimension $kN$. Moreover, the following
holds:

\begin{prop}(\cite{bsz}, Proposition 4.4)
\label{modelprops}
\begin{enumerate}
\item[(i)] For $\g\in \jetreg 1n$, the set of
  solutions $\cals_\g \subset \jetk nN$ is a linear
  subspace of codimension $kN$.
\item[(ii)] Set
\[\jetko nN=\left\{ \Psi \in \jetk nN|\dim\ker(\Psi')=1\right\}.\]
For any $\g \in \jetreg 1n$, the subset $\cals_\g \cap \jetko nN$ of
$\cals_\g$ is dense.
\item[(iii)] If $\Psi \in \jetko nN$, then $\Psi$ belongs
  to at most one of the spaces $\cals_\g$. More precisely,
\[
\text{if }\g_1,\g_2\in\jetreg 1n,\;\ \Psi\in \jetko nN \text{ and }
  \Psi\circ \g_1=\Psi\circ \g_2=0,
\]
then there exists
  $\vp \in \jetreg 11$ such that $\g_1=\g_2 \circ \vp$.
\item[(iv)] Given $\g_1,\g_2\in\jetreg 1n$, we have
  $\cals_{\g_1}=\cals_{\g_2}$ if and only if there is some
  $\vp \in \jetreg 11$ such that $\g_1=\g_2 \circ \vp$.
\end{enumerate}
  \end{prop}

By the second part of Proposition \ref{modelprops} we have a
well-defined map 
$$\n: \jetreg 1n \to \grass(\mathrm{codim}=kN,\jetk
nN), \,\,\,\, \g \mapsto \cals_\g$$ to the Grassmannian of codimension-$kN$
subspaces in $\jetk nN$. From the last part of Proposition
\ref{modelprops} it follows that:

\begin{prop}\label{propgrass}(\cite{bsz})
$\n$ is $\GG_k$-invariant on the $\jetreg 11$-orbits, and the induced map on
the orbits 
\begin{equation}\label{embeddinggrass}
\bar{\n}:\jetreg 1n/\GG_k \hookrightarrow
\grass(\mathrm{codim}=kN,\jetk nN)
\end{equation}
is injective.
\end{prop}

\section{Embedding into the flag of equations}\label{quotientflag}

In this section we will recast the embedding (\ref{embeddinggrass}) of
$\jetreg 1n/\GG_k$ given by Proposition \ref{propgrass} into a more useful form,
still following \cite{bsz}.
Let us rewrite the linear system $\Psi\circ \g=0$ associated to $\g
\in\jetreg 1n$ in a dual form. The system is based on the standard
composition map \eqref{comp}:
\[ \jetk nN\times\jetk 1n \longrightarrow\jetk 1N,
\]
which, via the identification $\jetk nN=\jetk n1\otimes\CC^N$, is
derived from the map
\begin{equation*}
 \jetk n1\times\jetk 1n \longrightarrow\jetk 11
\end{equation*}
via tensoring with $\CC^N$. Observing that composition is linear in
its first argument, and passing to linear duals, we may  rewrite
this correspondence in the form
\begin{equation}
  \label{mapson}
 \phi:\jetk 1n\longrightarrow \Hom(\jetk 11{}^*,\jetk n1{}^*).
\end{equation}

If $\g=(\g',\g'',\ldots ,\g^{(k)})\in \jetk 1n=(\CC^n)^k$ is the
$k$-jet of a curve, we can put $\g^{(j)}\in \CC^n$ into the $j$th
column of an $n\times k$ matrix, and
\begin{itemize}
\item identify  $\jetk 1n\text{ with }\Hom(\CC^k,\CC^n)$;
\item identify  $\jetk n1{}^* \text{ with }
\symdot=\oplus_{l=1}^k \Sym^l \CC^n$;
\item identify  $\jetk
11{}^* \text{ with } \CC^k$.
\end{itemize}

Using these identifications, we can recast the map $\phi$ in
\eqref{mapson} as
\begin{equation}
  \label{homs}
\phi_k:\Hom(\CC^k,\CC^n) \longrightarrow \Hom(\CC^k,\symdot),
\end{equation}
which may be written out explicitly as follows
\[   (\g', \g'', \ldots ,\g^{(k)})\longmapsto
\left(\g',\g''+(\g')^2, \dots, \sum_{i_1+i_2+\ldots
+i_s=d}\frac{1}{i_1!\ldots i_s!}\g^{(i_1)}\g^{(i_2)}\ldots
\g^{(i_s)}\right).
\]
The set of solutions $\mathcal{S}_\g$ is the linear subspace orthogonal to
the image of $\phi_k(\g', \ldots \g^{(k)})$ tensored by $\CC^N$; that
is,
\[\mathcal{S}_{\g}=\mathrm{im}(\phi_k(\g))^\perp \otimes \CC^N \subset \jetk nN.\]
Consequently, it is straightforward to take $N=1$ and define 
\begin{equation}\label{defsgamma}
\mathcal{S}_{\g}=\mathrm{im}(\phi_k(\g)) \in \grass(k, \symdot).
\end{equation} 

Moreover, let $B_k\subset GL(k)$ denote the Borel subgroup consisting of upper triangular matrices and let
\[  \flk=\Hom(\CC^k,\symdot)/B_k=\left\{
0=F_0\subset F_1\subset\dots\subset F_k\subset \CC^n,\; \dim
F_l=l\right\}
\]
denote the full flag of $k$-dimensional subspaces of $\symdot$.
In addition to \eqref{defsgamma} we can analogously define
\begin{equation}
\Fl_\g=(\imm(\phi(\g^1))\subset \imm(\phi(\g^2))\subset \ldots \subset \imm(\phi(\g^k)))\in \flag_k(\symdot) . 
\end{equation}

Using these definitions Proposition \ref{modelprops} implies the
the following version of Proposition \ref{propgrass}, which does not
contain the parameter $N$.

\begin{prop}\label{embedfinal} The map $\phi$ in \eqref{homs} is a
$\GG_k$-invariant algebraic morphism
\[ \phi: \jetreg 1n \rightarrow \Hom(\CC^k,\symdot),
\]
which induces 
\begin{itemize}
\item an injective map on the $\GG_k$-orbits to the
Grassmannian:
\[\phi^{Gr}: \jetreg 1n /\GG_k \hookrightarrow \grass(k, \symdot) \]
defined by $\phi^{Gr}(\g)=\mathcal{S}_\g$;
\item an injective map on the $\GG_k$-orbits to the
flag manifold:
\[\phi^{Flag}: \jetreg 1n /G_k \hookrightarrow \flag_k(\symdot) \]
defined by $\phi^{Flag}(\g)=\Fl_\g$.
\end{itemize}
In addition,
\[\phi^{Gr}=\phi^{Flag}\circ \pi_k\]
where $\pi_k:\flag(k,\symdot)\to \grass_k(\symdot)$ is the projection to the $k$-dimensional subspace.
\end{prop}

Composing $\phi^{Gr}$ with the Pl\"{u}cker embedding
\[\grass(k, \symdot) \hookrightarrow \PP(\wedge^k \symdot)\]
we get an embedding
\begin{equation}\label{embedproj}
\phi^{\Proj}:\jetreg 1n /\GG_k \hookrightarrow \PP(\wedge^k(\symdot)).
\end{equation}
The image $$\phi^{Gr}(\jetreg 1n)/\GG_k \,\, \subset \,\, \grass(k, \symdot)$$ is a
$GL(n)$-orbit in $\grass(k,\symdot)$, and therefore a nonsingular
quasi-projective variety. Its closure is, however, a highly singular
subvariety of $\grass(k, \symdot)$, which when $k \leq n$ is a finite union of $GL(n)$ orbits.  

\begin{defn} \label{defnxy} 
Recall that we can identify $J_k(1,n)$ with $\mathrm{Hom}(\CC^k,\CC^n)$ and then
$$\jetreg 1n = \{\rho \in \mathrm{Hom}(\CC^k,\CC^n): \rho(e_1) \neq 0 \}.$$
Let
$$\jetnondeg 1n = \{\rho \in \mathrm{Hom}(\CC^k,\CC^n): \mathrm{rank} \rho
= \mathrm{max} \{k,n\} \}$$
and let 
\[X_{n,k}=\phi^{\Proj}(\jetnondeg 1n),\,\,\, Y_{n,k}=\phi^{\Proj}(\jetreg 1n),\]
\end{defn}

so that if $n \leq k$ then 
\[X_{n,k} \subset Y_{n,k} \subset \grass(n,\symdott) \subset \PP(\wedge^k(\symdott)).\]

It is clear that $\jetnondeg 1n$ is an open subset of $\jetreg 1n$. If we identify the elements of $\jetk 1n$ with $n \times k$ matrices whose columns are the derivatives of the map germs $f=(f',\ldots ,f^{(n)}):\CC \to \CC^n$, then $\jetnondeg 1n$ is the set of such matrices of maximal rank and $\jetreg 1n$ consists of the matrices with nonzero first column.

\begin{definition} \label{defn6.1}
Let $e_1,\ldots, e_n$ be the standard basis of $\CC^n$; then
\[\{e_{i_1,i_2,\ldots, i_s}=e_{i_1} \ldots e_{i_s}:1\le i_1\le \ldots \le i_s \le n, 1 \le s \le k\}\] 
is a basis of $\symdott$, and 
\[\{e_{\e_1}\wedge \ldots \wedge e_{\e_n}:\e_l \in \Pi_{\le n}\}\]
is a basis of $\PP(\wedge^n(\symdott))$, where 
$$\Pi_{\le n}= \{(i_1,i_2,\ldots, i_s):1\le i_1\le \ldots \le i_s \le n, 1 \le s \le k\}.$$ 
The corresponding
coordinates of $x \in \symdott$ will be denoted by $x_{\e_1,\e_2,\ldots ,\e_d}$.
Let $A_{n,k}\subset \PP(\wedge^k(\symdott))$ consist of the points whose projection to $\wedge^k (\CC^n)$ is nonzero. This is the subset where $x_{i_1,i_2,\ldots, i_k} \neq 0$ for some $1 \le i_1 \le \ldots \le i_k \le n$. 
\end{definition}

\begin{rem}
If $n=k$ then $A_{n,n}\subset \PP(\wedge^k(\symdott))$ is the affine chart where $x_{1,2,\ldots,n}\neq 0$. 
\end{rem}

Let us take a closer look at the space $\grass(n,\symdott)$, which
has an induced $\GL(n)$ action coming from the $\GL(n)$
action on $\symdott$. 
Since $\phi^\Proj$ is a
$\GL(n)$-equivariant embedding, we conclude that

\begin{lemma} \label{6.3}
\begin{enumerate}
\item[(i)] For $k\le n$ $X_{n,k}$ is the $\GL(n)$
orbit of
\begin{equation}\label{zdisdef}
\zdis=\phi^{\Proj}(e_1,\ldots, e_k)=[e_1\wedge
(e_2+ e_1^2) \wedge \ldots \wedge
(\sum_{i_1+\ldots i_s=k}e_{i_1}\ldots e_{i_s})]
\end{equation}
in $\PP(\wedge^k(\symdott))$. For arbitrary $g\in \GL(n)$ with
column vectors $v_1,\ldots, v_n$ the action is given by
\[g\cdot \zdis=\phi^{\Proj}(g)=\phi^{\Proj}(v_1,\ldots, v_n)=[v_1\wedge (v_2
+ v_1^2) \wedge \ldots
\wedge (\sum_{i_1+\ldots +i_s=n}v_{i_1}\ldots v_{i_s})].\]
\item[(ii)] For $k\le n$ $Y_{n,k}$ is a finite union of $\GL(n)$ orbits.
\item[(iii)] For $k>n$ the images $X_{n,k}$ and $Y_{n,k}$ are $\GL(n)$-invariant quasi-projective varieties with no dense $\GL(n)$ orbit.   
\end{enumerate}
\end{lemma}

\begin{lemma}\label{unionoforbits} If $k\le n$ then  
\begin{enumerate}
\item[(i)] $A_{n,k}$ is invariant under the $\GL(n)$ action on $\PP(\wedge^k(\symdott))$. 
\item[(ii)] $X_{n,k} \subset A_{n,k}$; however, $Y_{n,k} \nsubseteq A_{n,k}$.
\end{enumerate}
\end{lemma}
\begin{proof}
To prove the first part take a lift
\[\tilde{z}=\tilde{z}^1\oplus \tilde{z}^2 \in \Hom(\CC^n,\symdott)\] 
of $z \in \grass(n,\symdott)$, where 
\[z^1 \in \Hom(\CC^n,\CC^n) \text{  and  } z^2 \in \Hom(\CC^n,\oplus_{i=2}^n \Sym^i(\CC^n))\] 
Then $z\in A_{n,k}$ if and only if
$x_{1,2,\ldots, n}(z)=\det(\tilde{z}^1) \neq 0$, which is preserved by
the $\GL(n)$ action.
For the second part note that for $(v_1, \ldots, v_k) \in \jetnondeg 1n$ we have
$v_1 \wedge \ldots, \wedge v_k \neq 0$ so by definition $\phi^\Proj(v_1,\ldots, v_k) \in A_{n,k}$. On the other hand 
\[\phi^\Proj(e_1,0,\ldots, 0)=e_1 \wedge e_1^2 \wedge \ldots \wedge e_1^k \in Y_{n,k}\setminus A_{n,k}.\] 
\end{proof}

When $k=n$ we have

\begin{lemma}
$X_{k,k} \cong \GL(k)/\GG_k$ is embedded in the affine space $A_{k,k} \subset \PP(\wedge^k \sym^{\le k}\CC^k)$ as the $\GL(k)$ orbit of $[e_1 \wedge (e_2+e_1^2) \wedge \ldots \wedge (\sum_{i_1+\ldots +i_s=k}e_{i_1}\ldots e_{i_s})].$
\end{lemma}

\section{Affine embeddings of $\SL(k)/\UU_k$}

In the last section we embedded $\GL(k)/\GG_k$ in the affine space $A_{k,k} \subset \PP(\wsymk)$ as the $\GL(k)$ orbit of 
\[[e_1 \wedge (e_2+e_1^2) \wedge \ldots \wedge (\sum_{i_1+\ldots +i_s=k}e_{i_1}\ldots e_{i_s})] \in \PP(\wsymk).\]
Equivalently we have $$\SL(k)/\SL(k) \cap \GG_k=\SL(k)/\UU_k \rtimes F_k$$ embedded in $\wsymk$ as the $\SL(k)$ orbit of 
\[p_k=e_1 \wedge (e_2+e_1^2) \wedge \ldots \wedge (\sum_{i_1+\ldots +i_s=k}e_{i_1}\ldots e_{i_s}),\]
where $\SL(k) \cap \GG_k$ is the semi-direct product $\UU_k \rtimes F_k$ of $\UU_k$ by the finite group $F_k$ of $\ell_k$th roots of unity in $\CC$ for $\ell_k=1+\ldots +k={k+1 \choose 2}$, embedded in $\SL(k)$ as
\[ \epsilon \mapsto \left(\begin{array}{cccc}\epsilon & 0 & \ldots & 0 \\ 0 & \epsilon^2 & \ldots & 0 \\  & & \ddots & \\ 0 & 0 & \ldots & \epsilon^{k}\end{array}\right)\in \SL(k).\]
In this section we will look for affine embeddings of $\SL(k)/\UU_k$ in
spaces of the form
\[ W_{k,K} = \wsymk \otimes (\CC^k)^{\otimes K} \]
for suitable $K$ and study their closures.

\begin{lemma}\label{mainlem}
Let $K=M(1+2+\ldots +k)+1={k+1 \choose 2}M+1$ where $M \in \mathbb{N}$. Then the point 
\[ p_k \otimes e_1^{\otimes K} \in \wsymk \otimes (\CC^k)^{ \otimes K} \]
where 
\[p_k=e_1 \wedge (e_2+e_1^2) \wedge \ldots \wedge (\sum_{i_1+\ldots +i_s=k}e_{i_1}\ldots e_{i_s})\in \wsymk\]
has stabiliser $\UU_k$ in $\SL(k)$. 
\end{lemma}

\begin{proof}
 By Proposition 5.1 the stabiliser of 
$$[p_k] \in \mathbb{P}(\wsymk) \cong \PP(\wsymk \otimes (\CC e_1)^{\otimes K})
\subseteq \PP(W_{k,K})$$ 
in $GL(k)$ is $\GG_k = \UU_k \rtimes \CC^*$, so the stabiliser of $$p_k \otimes 
e_1^{\otimes K} \in \wsymk \otimes (\CC^k)^{ \otimes K}$$ 
is contained in $\GG_k$. Moreover by the proof of Proposition 5.1 
the stabiliser of $p_k \otimes 
e_1^{\otimes K}$ contains $\UU_k$. Finally 
\[  \left(\begin{array}{cccc}z & 0 & \ldots & 0 \\ 0 & z^2 & \ldots & 0 \\  & & \ddots & \\ 0 & 0 & \ldots & z^{k}\end{array}\right) \in \CC^* \subseteq \GG_k\]
acts on  $p_k \otimes 
e_1^{\otimes K}$ as multiplication by 
$$z^{1+2+ \cdots + k + K} = z^{(M+1)(1+2+ \cdots +k) +1}$$
and has determinant 1 if and only if $z^{1+2+ \cdots + k}=1$, so it lies
in $\SL(k)$ and fixes  $p_k \otimes 
e_1^{\otimes K}$ if and only if $z=1$. 
\end{proof}

We will prove 
\begin{theorem}\label{mainthm}
If $k \geq 4$ and  $K=M(1+2+\ldots +k)+1$ where $M \in \mathbb{N}$ is
sufficiently large, then 
the orbit of  $p_k \otimes 
e_1^{\otimes K}$ 
where 
\[p_k=e_1 \wedge (e_2+e_1^2) \wedge \ldots \wedge (\sum_{i_1+\ldots +i_s=k}e_{i_1}\ldots e_{i_s})\in \wsymk\]
under the natural action of $\SL(k)$ on  
\[ W_{k,K} = \wsymk \otimes (\CC^k)^{\otimes K} \]
is isomorphic to $\SL(k)/\UU_k$, and its complement in its closure
$\overline{\SL(k)(p_k \otimes 
e_1^{\otimes K})}$ in $W_{k,K}$ 
has  codimension at least two. 
\end{theorem}
  
This theorem has an immediate corollary.
\begin{corollary} If $k \geq 2$ then 
$\UU_k$ is a Grosshans subgroup of $\SL(k)$, so that every linear action of $\UU_k$ which extends to a linear action of $\SL(k)$ has finitely generated invariants.
\end{corollary}
\begin{proof} This follows directly from Theorem \ref{mainthm} when $k \geq 4$.
When $k=2$ and $k=3$ it is already known (cf. \cite{rousseau}).
\end{proof}

The remainder of this section will be devoted to proving Theorem \ref{mainthm}.

\bigskip

It follows directly from Lemma \ref{mainlem} that the $\SL(k)$-orbit of  $p_k \otimes 
e_1^{\otimes K}$ 
in  
$ W_{k,K} = \wsymk \otimes (\CC^k)^{\otimes K}$
is isomorphic to $\SL(k)/\UU_k$. 

Recall that 
$$
\UU_k = \left\{ 
\left(\begin{array}{ccccc}
1 & \a_2 & \a_3 & \cdots  & \a_k \\
0        & 1 & 2\a_2 & \cdots &  2\a_{k-1}+ \ldots \\
0        & 0       & 1  & \cdots & 3\a_{k-2}+\ldots \\
\cdot    & \cdot   & \cdot    & \cdot &  \cdot \\
0 & 0 & \cdots & 1 & (k-1)\a_2 \\
0 & 0 & 0 & \cdots  & 1 
\end{array} \right): \a_2, \ldots, \a_k \in \CC \right\}
$$
so that $\UU_k$ is generated along its last column as well as along its
first row.

Let $B_k \subset \SL(k)$ denote the standard Borel subgroup of
$\SL(k)$ which stabilises the 
filtration $\CC e_1 \subset \CC e_1 \oplus \CC e_2 \subset \cdots \CC^k$.
Then $B_k=B_{k-1}\cdot \UU_k$ where the Borel subgroup $B_{k-1}$ of $\GL(k-1)=\GL(\CC e_1 \oplus \CC e_2 \oplus \ldots \oplus \CC e_{k-1})$ is embedded diagonally in
$\SL(k)$ via
$$A \mapsto \left(  \begin{array}{cc}
A & 0 \\
0 & (\mathrm{det}A)^{-1} 
\end{array}  
\right).$$
Since $\UU_k$ stabilises $p_k$ and $e_1$
we have 
\[\overline{B_k (p_k \otimes e_1^{\otimes K})}=\overline{B_{k-1} (p_k \otimes e_1^{\otimes K})},\]
and since $\SL(k)/B_k$ is projective we have
\[\overline{\SL(k) (p_k \otimes e_1^{\otimes K})}= \SL(k) \overline{B_k (p_k \otimes e_1^{\otimes K})} = \SL(k) \overline{B_{k-1} (p_k \otimes e_1^{\otimes K})}.\]
Since the closure $\overline{\SL(k) (p_k \otimes e_1^{\otimes K})}$ of the 
$\SL(k)$-orbit of $ p_k \otimes e_1^{\otimes K}$ in $W_{k,K}$ is the union of finitely many $\SL(k)$-orbits, to prove Theorem \ref{mainthm} it suffices to prove

\begin{lemma} \label{lemparts12}
Suppose that $k \geq 4$ 
and $a$ and $b$ are strictly positive integers with $b/a$ large enough
and that $x$ lies in the closure in
$$(\wsymk)^{\otimes a} \otimes (\CC^k)^{\otimes b}$$
of the orbit $B_k(p_k^{\otimes a} \otimes e_1^{\otimes b})$ of
$p_k^{\otimes a} \otimes e_1^{\otimes b}$ under the natural action of the  
Borel subgroup $B_k$ of $\SL(k)$. Then either $x \in B_k(p_k^{\otimes a} \otimes e_1^{\otimes b})$ or the stabiliser of $x$ in $\SL(k)$ has dimension at least $k+1$.
\end{lemma}

We will split the proof of this lemma into two parts. Let $T_k$ denote
the standard maximal torus of $\SL(k)$ consisting of the diagonal matrices
in $\SL(k)$. Lemma \ref{lemparts12} follows immediately from Lemmas \ref{lempart2} and \ref{lempart1} below.

\begin{lemma} \label{lempart2}
Suppose that $k \geq 4$  
and $a$ and $b$ are strictly positive integers with $b/a$ large enough
and that $x$ lies in the closure 
$\overline{T_k(p_k^{\otimes a} \otimes e_1^{\otimes b})}$ in
$$(\wsymk)^{\otimes a} \otimes (\CC^k)^{\otimes b}$$
of the orbit $T_k(p_k^{\otimes a} \otimes e_1^{\otimes b})$ of
$p_k^{\otimes a} \otimes e_1^{\otimes b}$ under the natural action of the  
maximal torus $T_k$ of $\SL(k)$. Then either $x \in T_k(p_k^{\otimes a} \otimes e_1^{\otimes b})$ or the stabiliser of $x$ in $\SL(k)$ has dimension at least $k+1$.
\end{lemma}

\begin{lemma} \label{lempart1}
Suppose that $k \geq 2$ 
and $a$ and $b$ are strictly positive integers
and that $x$ lies in the closure in
$$(\wsymk)^{\otimes a} \otimes (\CC^k)^{\otimes b}$$
of the orbit $B_k(p_k^{\otimes a} \otimes e_1^{\otimes b})$ of
$p_k^{\otimes a} \otimes e_1^{\otimes b}$ under the natural action of the  
Borel subgroup $B_k$ of $\SL(k)$. Then either $x \in B_k
\overline{T_k(p_k^{\otimes a} \otimes e_1^{\otimes b})}$ or the stabiliser of $x$ in $\SL(k)$ has dimension at least $k+1$.
\end{lemma}

We will start with the proof of Lemma \ref{lempart1}.

\begin{proof} We have 
$$x \in \overline{B_k (p_k^{\otimes a} \otimes e_1^{\otimes b}})=\overline{B_{k-1} (p_k^{\otimes a} \otimes e_1^{\otimes b})}$$
as above, so there is a sequence of matrices 
\[b^{(m)}=\left(\begin{array}{ccccc}b^{(m)}_{11} & b^{(m)}_{12} & \ldots & b^{(m)}_{1k-1}& 0 \\ 0 & b^{(m)}_{22} & \ldots & b^{(m)}_{2k-1}& 0 \\  & & \ddots 
& & \\ 0 & 0 & \ldots & 0 &  b^{(m)}_{kk}\end{array}\right)\in B_{k-1} \subset \SL(k)\]
such that $b^{(m)}(p_k^{\otimes a} \otimes e_1^{\otimes b}) \to x \text{ as } m\to \infty$. Now expanding the wedge product in the definition of 
$p_k$ we get
\[b^{(m)}(p_k^{\otimes a})=\left(e_1 \wedge \ldots \wedge e_n+\ldots +(b^{(m)}_{11})^{1+2+\ldots +k}e_1 \otimes e_1^2 \otimes \ldots \otimes e_1^k\right)^{\otimes a}\]
while
$$b^{(m)}(e_1^{\otimes b}) = (b^{(m)}_{11})^b e_1^{\otimes b},$$
so by considering the coefficient of $(e_1 \wedge \ldots \wedge e_n)^{\otimes a}
\otimes e_1^{\otimes b}$ we see that $(b^{(m)}_{11})^b$ tends to a limit in $\CC$
as $m \to \infty$. Thus, by replacing the sequence $(b^{(m)})$ with a subsequence
if necessary, we can assume that
$$ b^{(m)}_{11} \to b^{(\infty)}_{11} \in \CC$$
as $m \to \infty$.

First suppose that $k=2$. Then $\symkk = \CC^2 \oplus \sym^2 \CC^2$ and
$$(\wsymk)^{\otimes a} \otimes (\CC^k)^{\otimes b} = 
(\wedge^2(\CC^2 \oplus \sym^2 \CC^2))^{\otimes a} \otimes (\CC^2)^{\otimes b}$$
and
$$p_k = e_1 \wedge (e_2 + e_1^2),$$
so if
$$b^{(m)} = \left(\begin{array}{cc}
b^{(m)}_{11} & b^{(m)}_{12}\\ 0 & b^{(m)}_{22}
\end{array}
\right) \in \SL(2)$$
then $b^{(m)}_{11} b^{(m)}_{22} = 1$ and
$$b^{(m)}(p_2^{\otimes a} \otimes e_1^{\otimes b})
= (b^{(m)}_{11})^b (e_1 \wedge (e_2 + (b^{(m)}_{11})^3 e_1^2)))^{\otimes a} \otimes 
e_1^{\otimes b}$$
$$ \to x = (b^{(\infty)}_{11})^b (e_1 \wedge (e_2 + (b^{(\infty)}_{11})^3 e_1^2)))^{\otimes a} \otimes 
e_1^{\otimes b}$$
as $m \to \infty$. If $b^{(\infty)}_{11} \neq 0$
then $x \in \SL(2)((p_2^{\otimes a} \otimes e_1^{\otimes b})$, while
if $b^{(\infty)}_{11} = 0$ then $x=0$ is fixed by $\SL(2)$ which has dimension
$3=k+1$.

Now suppose that $k>2$, and assume first that $b^{(\infty)}_{11} \neq 0$.
We have that 
$$b^{(m)}(p_k^{\otimes a} \otimes 
e_1^{\otimes b}) = (b^{(m)}_{11})^b (
b^{(m)}p_k)^{\otimes a}) \otimes 
e_1^{\otimes b} \to x$$
 and 
$ b^{(m)}_{11} \to b^{(\infty)}_{11} \in \CC \setminus \{ 0 \}$
as $m \to \infty$, so by replacing the sequence $(b^{(m)})$ with a subsequence
if necessary, we can assume that
$$(b^{(m)}_{11})^{b/a} b^{(m)}p_k \to p_k^{\infty} \in \wsymk$$
as $m \to \infty$, where
\begin{multline}
b^{(m)}p_k = b^{(m)}_{11}e_1 \wedge(b^{(m)}_{22}e_2+(b^{(m)}_{11})^2e_1^2) \wedge \ldots 
\wedge 
(b^{(m)}_{ii}e_i+b^{(m)}_{i-1i}e_{i-1}+\ldots \\
\ldots +b^{(m)}_{1i}e_1+\sum_{s=2}^{i-1} \sum_{i_1+\cdots + i_s = i } 
(b^{(m)}_{i_1 i_1}e_{i_1} + \cdots + b^{(m)}_{1 i_1} e_1)
\ldots (b^{(m)}_{i_s i_s}e_{i_s}+\cdots +b^{(m)}_{1i_s}e_1)+(b^{(m)}_{11})^ie_1^i ) \wedge \ldots 
\end{multline}
Looking at the coefficient of 
\[e_1 \wedge e_1^2 \wedge \ldots \wedge e_1^{i-1} \wedge e_j \wedge e_1^{i+1} \wedge \ldots \wedge e_1^{k}\]
when $1 \leq j \leq i \leq k$, we see that 
\[(b^{(m)}_{11})^{1+2+\ldots +(i-1) + (i+1) +\ldots +k}b^{(m)}_{ji} \]
tends to a limit in $\CC$ as $m \to \infty$, and 
so since $b^{(\infty)}_{11} \neq 0$
\[b^{(m)}_{ji} \to b^{(\infty)}_{ji} \in \CC.\]
Also
$b^{(m)}_{11} b^{(m)}_{22} \cdots b^{(m)}_{kk} = 1$
for all $m$, so 
$b^{(\infty)}_{11} b^{(\infty)}_{22} \cdots b^{(\infty)}_{kk} = 1$,
so $b^{(m)} \to b^{(\infty)} \in \SL(k)$. Therefore 
\[x = b^{(\infty)}(p_k^{\otimes a} \otimes 
e_1^{\otimes b})\]
lies in the orbit of $p_k^{\otimes a} \otimes 
e_1^{\otimes b}$ as required.

So it remains to consider the case when $b^{(\infty)}_{11}=0$. 
If $p_k^\infty = 0$ then its stabiliser is $\SL(k)$ which has dimension
$k^2-1 \geq k+1$, so we can assume that $p_k^\infty \neq 0$. 
Recall that then 
$$(b^{(m)}_{11})^{b/a} b^{(m)}p_k \to p_k^{\infty} \in \wsymk$$
and
$$ [b^{(m)}p_k] \to [p_k^{\infty}] \in \PP(\wsymk)$$
as $m \to \infty$, where 
$$
b^{(m)}p_k = b^{(m)}_{11}e_1 \wedge(b^{(m)}_{22}e_2+(b^{(m)}_{11})^2e_1^2) \wedge \ldots 
\wedge 
(b^{(m)}_{ii}e_i+b^{(m)}_{i-1i}e_{i-1}+\ldots $$
$$
\ldots +b^{(m)}_{1i}e_1+\sum_{s=2}^{i-1} \sum_{i_1+\cdots + i_s = i } 
(b^{(m)}_{i_1 i_1}e_{i_1} + \cdots + b^{(m)}_{1 i_1} e_1)
\ldots (b^{(m)}_{i_s i_s}e_{i_s}+\cdots +b^{(m)}_{1i_s}e_1)+(b^{(m)}_{11})^ie_1^i ) \wedge \ldots 
$$
 By replacing the sequence $(b^{(m)})$ with a
subsequence if necessary, we can assume that
$$[b^{(m)}_{ii}e_i+b^{(m)}_{i-1i}e_{i-1}+
\ldots +b^{(m)}_{1i}e_1] \to [c^{(\infty)}_{ii}e_i+c^{(\infty)}_{i-1i}e_{i-1}+
\ldots +c^{(\infty)}_{1i}e_1] \in \PP(\CC^k)$$
as $m \to \infty$ for $2 \leq i \leq k$, which implies that 
$$
[(b^{(m)}_{i_1 i_1}e_{i_1} + \cdots + b^{(m)}_{1 i_1} e_1)
\ldots (b^{(m)}_{i_s i_s}e_{i_s}+\cdots +b^{(m)}_{1i_s}e_1)]\,\,\, \to  $$
$$ 
[(c^{(\infty)}_{i_1 i_1}e_{i_1} + \cdots + c^{(\infty)}_{1 i_1} e_1)
\ldots (c^{(\infty)}_{i_s i_s}e_{i_s}+\cdots +c^{(\infty)}_{1i_s}e_1)]
\,\, \in \,\, \PP(\mathrm{Sym}^i\CC^k)$$
whenever $i_1 + \cdots + i_s = i \in \{2,\ldots,k\}$, and hence that
$$p_k^{\infty} \in \wedge^k(\mathrm{Sym}^{\leq k}D)$$
where $D$ is the span in $\CC^k$ of 
$$\{ e_1\} \cup \{ c^{(\infty)}_{ii}e_i+c^{(\infty)}_{i-1i}e_{i-1}+
\ldots +c^{(\infty)}_{1i}e_1 : 2 \leq i \leq k \}.$$
Moreover since $b^{(m)} \in B_{k-1}$ we have $b^{(m)}_{jk} = 0$ if $j<k$
so
$$[c^{(\infty)}_{kk}e_k+c^{(\infty)}_{k-1k}e_{k-1}+
\ldots +c^{(\infty)}_{1k}e_1] = [ e_k]$$
so $e_k \in D$. 

Note that $b^{(m)} \in B_{k-1}$ and $B_{k-1}$ normalises the maximal
unipotent subgroup $U_k$ of $B_k$ which contains the stabiliser $\UU_k$ of
$p_k$. Therefore for each $m$ there is a $(k-1)$-dimensional subgroup of
$U_k$ which stabilises $b^{(m)}p_k$, and it follows that there is a 
 $(k-1)$-dimensional subgroup $\UU_k^{\infty}$ of
$U_k$ which stabilises $p_k^\infty$.
In addition by \cite{Birkes} Theorem 6.4 if $p_k^\infty$ does not lie
in $\SL(k)p_k$ then it is stabilised by a nontrivial one-parameter subgroup
$\lambda^\infty:\CC^* \to \SL(k)$ of $\SL(k)$.
 Moreover
if $D \neq \CC^k$ then there is some $j \in \{2,\ldots,k-1\}$
such that $e_j$ is not in $D$, and then there is an automorphism of $\CC^k$
which fixes every element of $D$ and sends $e_j$ to $e_j + e_k$. This automorphism is independent of $\UU_k^\infty$ (since $\UU_k^\infty \subseteq U_k$) and the one-parameter subgroup $\lambda^\infty$ of $\SL(k)$ fixing
$p_k^\infty$, so the stabiliser of $p_k^\infty$ in $\SL(k)$ has dimension
at least
$$\dim \UU_k^\infty + 2 = k+1.$$
Thus we can assume that $D=\CC^k$, and hence $c_{ii}^{(\infty)} \neq 0$
for $2 \leq i \leq k$, so that
$$\frac{b^{(m)}_{ji}}{b^{(m)}_{ii}} \to \frac{c^{(\infty)}_{ji}}{c^{(\infty)}_{ii}} \in \CC$$ 
as $m \to \infty$. Then by applying an element of $B_{k-1}$ to $p_k^\infty$ we can assume that
$$[c^{(\infty)}_{ii}e_i+c^{(\infty)}_{i-1i}e_{i-1}+
\ldots +c^{(\infty)}_{1i}e_1] = [ e_i]$$
or equivalently that
$$[b^{(m)}_{ii}e_i+b^{(m)}_{i-1i}e_{i-1}+
\ldots +b^{(m)}_{1i}e_1] \to [e_i] $$
as $m \to \infty$ for $2 \leq i \leq k$, and hence that 
$$
[(b^{(m)}_{i_1 i_1}e_{i_1} + \cdots + b^{(m)}_{1 i_1} e_1)
\ldots (b^{(m)}_{i_s i_s}e_{i_s}+\cdots +b^{(m)}_{1i_s}e_1)] \to 
[e_{i_1}\cdots e_{i_s}]
\in \PP(\mathrm{Sym}^i\CC^k)$$
whenever $i_1 + \cdots + i_s = i \in \{2,\ldots,k\}$. Now by again
replacing the sequence $(b^{(m)})$ with a
subsequence if necessary, we can assume that
$$[b^{(m)}_{ii}e_i+b^{(m)}_{i-1i}e_{i-1}+
\ldots +b^{(m)}_{1i}e_1
+\sum_{s=2}^{i-1} \sum_{i_1+\cdots + i_s = i } 
(b^{(m)}_{i_1 i_1}e_{i_1} + \cdots + b^{(m)}_{1 i_1} e_1] \to 
[d_i^\infty] \in \PP(\symkk)$$
where
$$d_i^\infty = \gamma_i^{(\infty)} e_i + 
\sum_{s=2}^{i} \sum_{i_1+\cdots + i_s = i } 
\gamma^{(\infty)}_{i_1\ldots i_s}e_{i_1} \cdots e_{i_s} \in \symkk \setminus \{ 0 \}$$
for some $\gamma^{(\infty)}_{i_1\ldots i_s} \in \CC$. In addition
$\{d_i^\infty: 1 \leq i \leq k \}$ is linearly independent so that
$$[p_k^\infty] = [d_i^\infty \wedge \cdots \wedge d_k^\infty ] \in \PP(\wsymk)$$
and $p_k^\infty = \lim_{m \to \infty} t^{(m)} p_k$ where
$t^{(m)}$ is the diagonal matrix with entries $b^{(m)}_{11},\ldots,b^{(m)}_{kk}$. 

Thus we can assume that $p_k^\infty \in \overline{T_k p_k}$ where $T_k$ is the standard maximal torus in $\SL(k)$, which completes the proof of Lemma
\ref{lempart1}.
\end{proof}

It therefore remains to prove Lemma \ref{lempart2}. 
We can continue with the notation above and use the following standard result:

\begin{lemma}\label{boundarylemma}
Let $T$ be an algebraic torus acting on the
projective variety $Z$, and $z\in Z$. Then $y\in \overline{Tz}$ if
and only if there is $\tau\in T$, and a one-parameter subgroup
$\l: \CC^* \to T$ such that $\tau y\in \overline{\l(\CC^*) z}$.
\end{lemma}

Hence we may assume without loss of generality that 
there is a one-parameter subgroup
$$ t \mapsto \lambda(t) = \left( \begin{array}{ccccc} t^{\lambda_1} & 0 & \cdots & & 0 \\
       0 & t^{\lambda_2} & 0 & \cdots & 0 \\
        & & \cdots & & \\
0 & \cdots & & 0 & t^{\lambda_k} 
\end{array} \right)$$
of $\SL(k)$ such that $\lambda_1 > 0$ and $t^{\lambda_1b/a}\lambda(t)p_k \to p_k^\infty$ as $t \to 0$. Therefore
$$p_k^\infty = \lim_{t\to 0} t^{\lambda_1b/a} e_1 \wedge (e_2 + t^{2\lambda_1 - \lambda_2}e_1^2) \wedge \cdots\wedge (e_k + \sum_{s=2}^k \sum{i_1 + \cdots + i_s = k} t^{\lambda_{i_1} + \cdots + \lambda_{i_s} - \lambda_k} e_{i_1} \cdots e_{i_s})$$
where $\lambda_1 + \cdots + \lambda_k = 0$.  We are assuming that $p_k^\infty \neq 0$ so
$$[p_k^\infty] = \lim_{t\to 0} [e_1 \wedge (e_2 + t^{2\lambda_1 - \lambda_2}e_1^2) \wedge \cdots\wedge (e_k + \sum_{s=2}^k \sum{i_1 + \cdots + i_s = k} t^{\lambda_{i_1} + \cdots + \lambda_{i_s} - \lambda_k} e_{i_1} \cdots e_{i_s})].$$
If $\lambda_{i_1} + \cdots + \lambda_{i_s} < \lambda_j$ for some $j \in \{2,\ldots, k-1 \}$ 
and $s \geq 2$ and $i_1, \ldots, i_s \geq 1$ such that $i_1 + \cdots + i_s = j$, then $[p_k^\infty]$ is independent of $e_j$ and so as above the stabiliser of $p_k^\infty$ in $\SL(k)$ has dimension at least $k+1$. So we can assume that
\begin{equation} \label{canassume} \lambda_{i_1} + \cdots + \lambda_{i_s} \geq \lambda_j \end{equation} for any $j \in \{2,\ldots, k-1 \}$ 
and $s \geq 2$ and $i_1, \ldots, i_s \geq 1$ such that $i_1 + \cdots + i_s = j$,
and in particular that $\lambda_j \leq j \lambda_1$ for each $j \in \{2,\ldots, k-1 \}$.
Let
\begin{equation} \label{rhodefn} \rho_j = j\lambda_1 - \lambda_j \end{equation}
for $j \in \{1,\ldots, k-1 \}$; then $\rho_1 = 0$ and $\rho_j \geq 0$ and 
$$\rho_{i_1} + \cdots + \rho_{i_s} \leq \rho_j$$ for any $j \in \{2,\ldots, k-1 \}$ 
and $s \geq 2$ and $i_1, \ldots, i_s \geq 1$ such that $i_1 + \cdots + i_s = j$.
In addition looking at the coefficient of 
$$e_1 \wedge e_2 \wedge \cdots \wedge e_{k-1} \wedge e_{i_1} \cdots e_{i_s}$$
where $i_1 + \cdots + i_s = k$, we find that 
$$0 \leq \lambda_1 b/a + \lambda_{i_1} + \dots + \lambda_{i_s} - \lambda_k = \lambda_1(b/a + k(k+1)/2) -
(\rho_{i_1} + \cdots + \rho_{i_s} + \rho_2 + \cdots + \rho_{k-1}),$$
and since $p_k^{\infty} \neq 0$ there is some $i_1, \ldots, i_s$ with
$i_1 + \cdots + i_s = k$ and  
\begin{equation} \label{page25} \lambda_1 b/a + \lambda_{i_1} + \dots + \lambda_{i_s} = \lambda_k \end{equation}
or equivalently
$$ \lambda_1(b/a + k(k+1)/2) =
\rho_{i_1} + \cdots + \rho_{i_s} + \rho_2 + \cdots + \rho_{k-1}.$$
Thus
\begin{equation} \label{hey}
p_k^\infty = \lim_{t\to 0}  e_1 \wedge (e_2 + t^{2\lambda_1 - \lambda_2}e_1^2) \wedge \cdots \end{equation} 
$$ \cdots \wedge (e_{k-1} + \sum_{s=2}^{k-1} \sum_{i_1 + \cdots + i_s = k-1} t^{\lambda_{i_1} + \cdots + \lambda_{i_s} - \lambda_{k-1}} e_{i_1} \cdots e_{i_s})
\wedge (t^{\lambda_1b/a}\sum_{s=2}^{k} \sum_{i_1 + \cdots + i_s = k} t^{\lambda_{i_1} + \cdots + r_{i_s} - r_{k}} e_{i_1} \cdots e_{i_s}$$
$$=  e_1 \wedge  \cdots\wedge (e_{k-1} + 
\sum_{s=2}^{k-1} \sum_{i_1 + \cdots + i_s = k-1:\rho_{i_1} + \cdots + \rho_{i_s} = \rho_{k-1}} e_{i_1} \cdots e_{i_s})
\wedge 
$$
$$(\sum_{s=2}^{k} \sum_{\substack{i_1 + \cdots + i_s = k\\
\lambda_1(b/a + k(k+1)/2) =
\rho_{i_1} + \cdots + \rho_{i_s} + \rho_2 + \cdots + \rho_{k-1}
}} e_{i_1} \cdots e_{i_s})
$$
is independent of $e_k$ and hence is fixed by the automorphisms of $\CC^k$ which
fix $e_1, \ldots, e_{k-1}$ and send $e_k$ to $e_k + e_j$ for $j \in \{1, \ldots, k-1\}$, as well as by the one-parameter subgroup 
$$ \lambda(t) = \left( \begin{array}{ccccc} t^{\lambda_1} & 0 & \cdots & & 0 \\
       0 & t^{\lambda_2} & 0 & \cdots & 0 \\
        & & \cdots & & \\
0 & \cdots & & 0 & t^{\lambda_k} 
\end{array} \right)$$
of $T_k$. 
Thus to complete the proof of Lemma \ref{lempart2} and hence of Theorem \ref{mainthm},
it suffices to find an additional one-dimensional stabiliser, which will be
done in the rest of this section.

Letting 
\[ \zdis= [p_k] = [ e_1\wedge ( e_2 + e_1^2) \wedge \ldots \wedge (\sum_{i_1+\ldots +i_s=k}e_{i_1}\ldots e_{i_s})]\]
as at (\ref{zdisdef}) we have
\begin{eqnarray*}
\l(t)\zdis=[t^{\l_1}e_1 \wedge (t^{\l_2}e_2+ t^{2\l_1}e_1^2)
\wedge \ldots \wedge (\sum_{i_1+\ldots +i_s=k}
t^{\l_{i_1}+\ldots +\l_{i_s}} e_{i_1}\ldots e_{i_s})]=\\
=[t^{\l_1+\ldots +\l_k}(e_1\wedge \ldots \wedge e_k) +
t^{\l_1+2\l_1+\l_3+\ldots +\l_k}(e_1 \wedge e_1^2 \wedge e_3 \wedge
\ldots \wedge e_k) + \ldots ].
\end{eqnarray*}
The generic term in this expression is
\[t^{\l_{\e_1}+\l_{\e_2}+\ldots \l_{\e_k}} (\be_{\e_1} \wedge \ldots \wedge \be_{\e_k}),\ \Sigma(\e_i)=i\] 
where
\begin{equation}\label{eform}
\l_\tau=\sum_{i\in \tau}\l_i \text{ and } \be_\tau=\Pi_{i\in \tau}e_i \text{ if } \tau=(i_1,\ldots, i_s). 
\end{equation}

\begin{defn} \label{4.1plus} For any one-parameter subgroup $\lambda$ as above
let
\begin{itemize}
\item $m_{\l}=\min_{\substack{(\e_1,\ldots \e_k)\\ 1 \leq \Sigma(\e_i) \leq k}}(\l_{\e_1}+\l_{\e_2}+\ldots
\l_{\e_k})$,
\item $\zdis_\l=[\sum_{1 \leq \Sigma \e \leq k, \l_{\e}=m_{\l}}\be_\e]$,
\item  $m_{\l}[i]=\min_{\Sigma(\e)=i}\l_{\e}$ for $1\le i \le k$,  
\item $\zdis_\l[i]=[\sum_{\Sigma \e=i, \l_{\e}=m_{\l}[i]}\be_\e]$.
\end{itemize}
Let $\mathcal{O}_\l$ denote the $\SL(k)$-orbit of
$\zdis_\l$. 
\end{defn}

 It is clear that the one-parameter subgroup $\tilde{\l}(t)=(t,t^2,\ldots, t^k)$ stabilises $\zdis$, where $\zdis$ is defined as at
\eqref{zdisdef}, and therefore $\zdis=\bz_{\tilde{\l}}$
and its $\SL(k)$-orbit is equal to its $\GL(k)$-orbit.

We need a more precise description of the orbit structure
of the closure of the orbit
$\calo_0=\calo_{\tilde{\l}}$. Since $\tilde{\l}_i=i\tilde{\l}_1$ for $i=1,\ldots ,k$, for $\l \neq \tilde{\l}$ we have a smallest index $ \s \in \{2, \ldots, k\}$ with $\l_\s \neq \s \l_1$. 
\begin{definition}
We call $\s=Head(\l)$ the head of $\l=(\l_1,\ldots, \l_n)$
if
\[\l_i=i\l_1 \text{ for } i<\s \text { and } \l_\s \neq \s \l_1.\]
If $\l_\s<\s \l_1$ then we call $\l$ {\em regular }; otherwise we call
$\l$ {\em degenerate}. 
\end{definition}

We will say that a one-parameter subgroup $\l$ is {\em maximal}
if the closure of the orbit $\GL(k)\cdot \bz_\l$ is a maximal boundary component of the closure of the orbit of $\bz$. 


\begin{definition}
Fix $0< \vare <1$ and $2 \le \s \le k$. Let $\l^{\s}=(\l_1^\s,\ldots
,\l_k^\s)$ and $\mu^{\s}=(\mu_1^\s,\ldots, \mu_k^\s)$ be the following one-parameter subgroups of $\GL(k)$:
\begin{equation}
\l_i^\s= 
 i-\lfloor \frac{i}{\s} \rfloor \vare \text{ for } 1\le i\le k, 
\end{equation}
\begin{equation}
\mu_i^\s=\begin{cases} i \text{ for } i \neq \s, i\le k, \\ \s+\varepsilon \text{ for } i=\s. 
 \end{cases}
\end{equation}
\end{definition}

It is easy to see that $\mathrm{Head}(\l^{\s})=\mathrm{Head}(\mu^\s)=\s$, and $\l^\s$ is regular, whereas $\mu^\s$ is degenerate.

\begin{definition}\label{defdimension}
Let $\l$ be a 1-parameter subgroup. We call 
\[\sharp\{i:\bz_\l[i]=e_i\}\]
the toral dimension of the limit point $\bz_\l$.  
\end{definition}

\begin{lemma}\label{largestorbits} 
If the $\SL(k)$-orbit of $p_k^\infty$ has codimension one in 
$\overline{\SL(k)p_k}$, then $[p_k^\infty]$ lies in the orbit of one of
$\zdis_{\l^2},\ldots , \zdis_{\l^k}$ or 
$\zdis_{\mu^2},\ldots, \zdis_{\mu^{k-1}}$.
\end{lemma}

\begin{proof}
We can assume that $[p_k^\infty ] = \zdis_\l$ for some one-parameter subgroup $\l$. First
suppose that $\l$ is a regular one-parameter subgroup with $\mathrm{Head}(\l)=\s$ and $[p_k^\infty ] = \zdis_\l$. Without loss of generality we can assume that 
\[\l_i=i \text{ for } i <\s \text{ and } \l_\s=\s-\varepsilon.\]
We will call $d(i)=\lfloor \frac{i}{\s} \rfloor$ the defect of $i$ and $d(\tau)=d(i_1)+\ldots +d(i_s)$ the defect of $\tau=(i_1,\ldots, i_s)$, so that when $i \leq \s$ we have
$d(i) \epsilon=\rho_i$ as defined at \eqref{rhodefn}.
Since 
\[\l_{(j,\underbrace{\s,\ldots, \s}_{m})}=j+m(\s-\varepsilon) \text{ for } 1 \le j \le \s-1, m \ge 0,\]
we have  
\begin{equation}\label{minweight}
m_\l[i]\le i-d(i)\varepsilon \text{ for } 1 \le i \le k.
\end{equation}
If $\l_s < s-d(s)\varepsilon$ for $s>i$ and $s$ is the smallest index with this property then $m_\l[s]=\l_s$ and $\bz_\l[s]=e_s$, so 
\[\bz_\l[1]=e_1,\bz_\l[\s]=e_\s,\bz_\l[s]=e_s,\]
while $\zdis_\l$ is independent of $e_k$ by (\ref{hey}), so $[p_k^\infty]$
is fixed by a three-dimensional torus in $\SL(k)$ and thus
$p_k^\infty$ is fixed by a two-dimensional torus in $\SL(k)$ as well as
a unipotent subgroup of dimension $k-1$.
So we can assume that $\l_i \ge i-d(i)\varepsilon$ for $1\le i \le k$, and therefore
\[m_\l[i]=i-d(i)\varepsilon \text{ for } 1\le i \le k. \]
So  
\begin{equation}\label{notin}
\be_\tau \notin \bz_\l[i] \text{ if } d(\tau) > d(i).
\end{equation}
On the other hand the distinguished 1-parameter subgroup $\l^\s$ is defined as $\l_i^\s=i-d(i)\varepsilon$, and therefore  
\begin{equation}\label{ithcomponent}
\bz_{\l^\s}[i]=\sum_{\Sigma(\tau)=i, d(\tau)=d(i)}\be_\tau.
\end{equation}
Comparing \eqref{notin} and \eqref{ithcomponent} we conclude
\[\bz_{\l}[i] \subset \bz_{\l^\s}[i] \text { for } 1\le i \le k.\]
Now let $\mu$ be a degenerate 1-parameter subgroup with $\mathrm{Head}(\mu)=\s$. Without loss of generality we can assume again that 
\[\mu_i=i \text{ for } i <\s \text{ and } \mu_\s=\s+\varepsilon.\]
Since
\[\mu_{(\underbrace{1,\ldots 1}_{i})}=i \text{ for } 1 \le i \le k\]
we have 
\begin{equation}\label{minweightmu}
m_\mu[i]\le i. 
\end{equation}
Again, $\mu_s<s$ cannot happen for $s>\s$ since in that case $\bz_\mu[s]=e_s$ would hold and 
the codimension of $\SL(k)p_k^\infty$ would be at least two.
 So $\mu_s\ge s$ and therefore $\mu_\tau \ge \Sigma(\tau)$ with strict inequality if $\s \in \tau$. Therefore  
\begin{equation}\label{compone}
\be_\tau \notin \bz_\mu[i] \text{ if } \s \in \tau.
\end{equation}
On the other hand $\mu^\s$ satisfies equality in \eqref{minweightmu}, and   
\begin{equation}\label{comptwo}
\bz_{\mu^\s}[i]=\sum_{\Sigma(\tau)=i, \s \notin \tau}\be_\tau.
\end{equation}
Comparing \eqref{compone} and \eqref{comptwo} we get 
\[\bz_{\mu}[i] \subset \bz_{\mu^\s}[i] \text { for } 1\le i \le k,\]
and so it remains to consider the possibility that $[p_k^\infty] =
\zdis_{\mu^k}$. But by \eqref{page25} 
there is some $i_1, \ldots, i_s$ with
$i_1 + \cdots + i_s = k$ and  
$$ \lambda_1 b/a + \lambda_{i_1} + \dots + \lambda_{i_s} = \lambda_k $$
and hence
$\lambda_k > \lambda_{i_1}+\ldots +\lambda_{i_s}$.  Thus $[p_k^\infty]$ cannot be equal to $\zdis_{\mu^k}$ because the coefficient of $e_1 \wedge e_1^2 \ldots \wedge e_1^k$ is nonzero for $\zdis_{\mu^k}$ but zero for $[p_k^\infty]$,
and the result 
 follows.
\end{proof}

We summarize our information about the maximal boundary components in 
\begin{prop}\label{exactform}
 We have $\bz_{\l^\s}=\wedge_{i=1}^k\bz_{\l^\s}[i]$, 
where $\bz_{\l^\s}[i]=\oplus_{\Sigma(\tau)=i,d(\tau)=d(i)}\be_\tau$, and  
 $\bz_{\mu^\s}=\wedge_{i=1}^k\bz_{\mu^\s}[i]$ where $\bz_{\mu^\s}[i]=\oplus_{\Sigma(\tau)=i,\s \notin \tau}\be_\tau$.
\end{prop}

\begin{rem}\label{remark} Since the  one-parameter subgroup $\tilde{\l}(t)=(t,t^2,\ldots, t^k)$ of $\GL(k)$ stabilises $T_k \zdis$,
it follows from Lemma \ref{largestorbits} that it is
enough to prove the codimension-at-least-two property we require only for the 
one-parameter subgroups $\tilde{\l}^\s$ 
(for $2 \leq s \leq k$) and $\tilde{\mu}^\s$ 
(for $2 \leq s \leq k-1$) of $\SL(k)$ given by 
$$\tilde{\l}^\s(t) = (\l^\s(t) \tilde{\l}(t)^{q_\s})^{n_\s}$$
and
$$\tilde{\mu}^\s(t) = (\mu^\s(t) \tilde{\l}(t)^{r_\s})^{m_\s}
$$ for suitable $q_\s, r_\s \in \QQ$ and $n_\s,m_\s \in \ZZ$. But we observed at
(\ref{canassume}) that the property is satisfied by a one-parameter subgroup
$\lambda$ of $\SL(k)$ if
$ \lambda_{i_1} + \cdots + \lambda_{i_s} < \lambda_j$ for any $j \in \{2,\ldots, k-1 \}$ such that 
$ {i_1} + \cdots + {i_s} = j$,
so it is enough to consider the one-parameter subgroups $\tilde{\l}^\s$
for $2 \leq s \leq k$.  
\end{rem}

\subsection{The limit of the stabilisers}

In order to prove Lemma \ref{lempart2}, it  now suffices by  Remark \ref{remark}
to  find a $k$-dimensional unipotent subgroup of
the stabiliser $G_{\zdis_{\l^\s}}$ of $\zdis_{\l^\s}$ 
in $\GL(k)$ for each $\s$
when $\zdis_{\l^\s}=[p_k^\infty]$, since we know that $p_k^\infty$ is fixed
by a one-parameter subgroup of the maximal torus $T_k$ of
$\SL(k)$, and any unipotent group which stabilises $\zdis_{\l^\s}=[p_k^\infty]$
also stabilises $p_k^\infty$.

In this subsection we will study the limits $\lim G_{\l^\s(t)\bz}$ 
of the stabiliser groups for the one-parameter subgroups $\l^\s$ 
 for $2\le \s \le k$, and use this to prove Lemma \ref{lempart2}, which together with Lemma \ref{lempart1} will complete the proof
of  Theorem \ref{mainthm}.

\begin{prop}\label{lemma1}
$G^\s=\lim_{t \to 0} G_{\l^{\s}(t)\zdis} \subset GL(k)$ is a
$k$-dimensional subgroup of $G_{\zdis_{\l^{\s}}}$ which contains
a $k-1$-dimensional subgroup of the maximal unipotent subgroup
$U_k$ of $\SL(k)$.
\end{prop}

\begin{proof}

Consider the stabilizer 
\[G_{\l^{\s}(t)\zdis}=\l^{\s}(t)^{-1}G_{\zdis} \l^{\s}(t).\]
Recall that 
\[G_\bz= \left\{
\left(
\begin{array}{ccccc}
\alpha_1 & \alpha_2   & \alpha_3          & \ldots & \alpha_k\\
0        & \alpha_1^2 & 2\alpha_1\alpha_2 & \ldots & 2\alpha_1\alpha_{n-1}+\ldots \\
0        & 0          & \alpha_1^3        & \ldots & 3\alpha_1^2\alpha_ {k-2}+ \ldots \\
0        & 0          & 0                 & \ldots & \cdot \\
\cdot    & \cdot   & \cdot    & \ldots & \alpha_1^d
\end{array}
 \right) \right\}
\]
where the polynomial in the $(i,j)$ entry is
\[p_{i,j}(\ba)=\sum_{a_1+a_2+\ldots +a_i=j}\a_{a_1}\a_{a_2} \ldots \a_{a_i}.\]
Therefore, the $(i,j)$ entry of the stabilizer of $\l^s(t)\zdis$ is
\begin{equation}
(G_{\l^{\s}(t)\zdis})_{i,j}=t^{\l_i^\s-\l_j^\s}p_{i,j}(\a)
\end{equation}

If $\vare$ is small enough then $\l_1^\s < \l_2^\s < \ldots < \l_k^\s$,
and we define the positive number
\begin{equation}
n_i^\s=\max_{1\le j\le n-i+1}(\l_{j+i-1}^\s-\l_j^\s),\ i=1,\ldots ,k.
\end{equation}

Note that by definition $n_1^\s=0$ for all $\s$.

\begin{lemma}\label{limitstab}
Under the substitution
\[\b_i^\s=t^{-n_i^\s}\a_i^\s\]
we have
\[G_{\l^{\s}(t)\zdis}(\b_1,\ldots ,\b_k)\in
GL(\CC[\b_1,\ldots, \b_k][t]),\] so the entries are polynomials in
$t$ with  coefficients in $\CC[\b_1,\ldots, \b_k]$.
\end{lemma}

\begin{proof} Compute the substitution as follows:

\begin{eqnarray}\label{substitute}
G_{\l^{\s}(t)\zdis})_{i,j}=t^{\l_i^\s-\l_j^\s}\sum_{a_1+a_2+\ldots
+a_i=j}\a_{a_1}\a_{a_2} \ldots
\a_{a_i}=\\
=\sum_{a_1+\ldots
a_i=j}t^{\l_i^\s-\l_j^\s}t^{n_{a_1}^\s+n_{a_2}^\s+\ldots
+n_{a_i}^\s}\b_{a_1} \b_{a_2} \ldots \b_{a_i}.
\end{eqnarray}
By definition
\[n_{a_1}^\s \ge \l_{i+a_1-1}^\s-\l_{i}^\s;\ n_{a_2}^\s \ge
\l_{i+a_1+a_2-2}^\s-\l_{i+a_1-1}^\s;\ \ldots \ ;\ n_{a_j}^\s \ge
\l_{i+a_1+\ldots +a_i-i}^\s-\l_{i+a_1+\ldots +a_{i-1}-(i-1)}^\s.
\]
Adding up these inequalites and using $a_1+\ldots +a_i=j$ we get an
alternating sum on the left cancelling up to
\[n_{a_1}^\s+\ldots +n_{a_i}^\s \ge \l_j^\s-\l_i^\s.\]
Substituting this into \eqref{substitute} we get
\begin{equation}
(G_{\l^{\s}(t)\zdis})_{i,j}=\sum_{a_1+\ldots a_i=j}
t^{\l_i^\s-\l_j^\s}t^{n_{a_1}^\s+n_{a_2}^\s+\ldots
+n_{a_i}^\s}\b_{a_1} \b_{a_2} \ldots \b_{a_i} \in \CC[\b_1,\ldots
,\b_k][t].
\end{equation}
This proves Lemma \ref{limitstab}.
\end{proof}

As a corollary we get the existence of 
\[G^\s=\lim_{t \to 0}G_{\l^{\s}(t)\zdis}(\b_1,\ldots ,\b_k) \in
GL(\CC[\b_1,\ldots ,\b_k]).\] To prove that $\dim G^\s=k$ and
complete the proof of Proposition \ref{lemma1}, for $1 \le i \le k$ choose $\t(i)$
such that
\begin{equation}\label{deftheta}
n_i^\s=\l_{\t(i)+i-1}-\l_{\t(i)}
\end{equation}
holds. Then
\begin{eqnarray}
p_{\t(i),\t(i)+i-1}(\b_1,\ldots ,\b_k)=\sum_{a_1+\ldots
+a_{\t(i)}=\t(i)+i-1}t^{n_{a_1}^\s+\ldots +n_{a_\t(i)}^\s}\b_{a_1}
\ldots \b_{a_{\t(i)}}
\end{eqnarray}
so
\begin{multline}\label{limitentries}
(G^\s)_{\t(i),\t(i)+i-1}=\lim_{t \to 0}
t^{-n_i^\s}p_{\t(i),\t(i)+i-1}(\b_1,\ldots ,\b_k)=\lim_{t \to 0}
(t^{n_i^\s}\b_1^{\t(i)-1}\b_i+\ldots )=\\
=\b_1^{\t(i)-1}\b_i+q_{\t(i),\t(i)+i-1}
\end{multline}
where
\[q_{\t(i),\t(i)+i-1} \in \CC[\b_1,\ldots, \b_k][t].\]
It follows that the elements $\frac{d}{dt}A^\s(t(e_1+e_i)1) \in
\mathrm{Lie}(G^\s)$ are independent, where $t(e_1+e_i)=(t,0,\ldots
,0,t,0,\ldots ,0)$ with the $t$'s are in the $1$st and $i$th
position if $i>1$ but interpreted as $(2t,0,\ldots ,0)$ if $i=1$.
This completes the proof of
Proposition \ref{lemma1}.
\end{proof}

In order to prove Lemma \ref{lempart2}, it  now suffices to  find an extra one-dimensional unipotent subgroup of
the stabiliser $G_{\zdis_{\l^\s}}$ of $\zdis_{\l^\s}$ for each $\s$
when $\zdis_{\l^\s}=[p_k^\infty]$, since we know that $p_k^\infty$ is fixed
by a one-parameter subgroup of the maximal torus $T_k$ of
$\SL(k)$ and by a $k-1$-dimensional unipotent subgroup of $G^\s=\lim_{\t \to
0}G_{\l^\s(t)\zdis}$ which is contained in
the standard maximal unipotent subgroup $U_k$ of $\SL(k)$. It turns out that we have to distinguish three cases here.

\bigskip 

\noindent {\bf Case 1:} 
 $\s=k$.

\begin{proof}

Let $T_\zeta \in GL(k)$ denote the transformation
\[T_\zeta (e_i)=e_i \text{ for } i\neq k-1\ ;\ T_\zeta(e_{k-1})=e_{k-1}+\zeta e_k 
\text{ for } \zeta \in \CC.\]
Since $e_{k-1}$ does not occur just in $\zdis_{\l^\s}[k-1]$, 
$T_\zeta$ stabilises $p_k^\infty$. This gives
us a subgroup of $\SL(k)$ of dimension at least $k+1$
which stabilises $p_k^\infty$, because $T_\zeta$ is unipotent but
not upper triangular if $\zeta \neq 0$.
\end{proof}

\bigskip 

\noindent {\bf Case 2:} 
 $\s < k$ and $k\neq -1$ mod $\s.$

\begin{proof}

Let $T$ be the transformation
\begin{equation}
T(e_i)=e_i \text{ for } i\neq k\ ;\ T(e_k)=e_k+\zeta e_\s.
\end{equation}
Since $e_k$ occurs only in $\zdis_{\l^\s}[k]$, and
$\zdis_{\l^\s}[\s]=\s$, 
 we have
\begin{multline}
T\cdot \zdis_{\l^\s}=\zdis_{\l^\s}(e_1,\ldots, e_{k-1},e_k+\zeta
e_{\s})=\\
=\zdis_{\l^\s}[1] \wedge \ldots \wedge \zdis_{\l^\s}[\s-1] \wedge
e_\s \wedge \zdis_{\l^\s}[\s+1] \wedge \ldots \wedge
\zdis_{\l^\s}[k])+ \\
+\zeta \cdot \zdis_{\l^\s}[1] \wedge \ldots \wedge
\zdis_{\l^\s}[\s-1] \wedge e_\s \wedge \zdis_{\l^\s}[\s+1] \wedge
\ldots \wedge \zdis_{\l^\s}[k-1] \wedge e_\s = \zdis_{\l^\s},
\end{multline}
so $T\in G_{\zdis_{\l^\s}}$.

It is slightly harder task to show that $T \not\in G^\s=\lim_{\t \to
0}G_{\l^\s(t)\zdis}$. First, we compute $n_i$ for $i=k-\s$. We claim
that for $k \neq -1$ mod $\s$
\begin{equation}\label{maxdifference}
n_{k-\s+1}=\l_k^\s-\l_\s^\s=\l_{k-\s+1}^{\s}-\l_1^\s.
\end{equation}
Indeed,
\[\l_{j+k-\s-1}-\l_j = \ldots \text{... } \le \l_k^\s-\l_\s^\s=\l_{k-\s+1}^{\s}-\l_1^\s\]
This means that we can choose $\t(k-\s+1)=\s$ in \eqref{deftheta}
and substitute into \eqref{limitentries}
\begin{equation}\label{sigmadentry}
(G^\s)_{\s,k}=\b_1^{\s-1}\b_{k-\s+1}+q_{\s,k}(\b_1,\ldots, \b_k),
\end{equation}
where $q_{\s,k}(\b_1,\ldots, \b_k)$ is a polynomial, whose monomials
$\b_{i_1}^{b_1}\ldots \b_{i_\s}^{b_\s}$ satisfy
\begin{equation}
i_1b_1+\ldots +i_\s b_\s=k.
\end{equation}
Moreover, we can also choose $\t(k-\s+1)=1$, by
\eqref{maxdifference}, and then \eqref{limitentries} gives us
\begin{equation}\label{firstrow}
(G^\s)_{1,k-\s+1}=\b_{k-\s+1}.
\end{equation}

Suppose now that $T\in G^\s$, that is
\begin{equation}\label{aandtequal}
T=G^\s(\b_1,\ldots, \b_k) \text{ for some } \b_1\in \CC^*,
\b_2,\ldots, b_k \in \CC.
\end{equation}
Let $(T)_{i,j}$ denote the $(i,j)$ entry of $T$. Then
\[(T)_{\s,k}=\zeta \text{ , } (T)_{i,j}=0 \text{ for } i\neq j\ ,\ (T)_{i,i}=1.\]
Comparing the $(1,1)$ and $(1,k-\s+1)$ entries of $T$ and $G^\s$ we
get
\begin{equation}\label{knownbetas}
\b_1=1,\b_{\d-\s+1}=0.
\end{equation}
 Choose $\t(i)$ for $i=2,\ldots ,k$ as in
\eqref{deftheta} and let $\t(k-\s+1)=\s$. Since all off-diagonal
entries of $T$ but the $(\s,k)$ are zero, \eqref{aandtequal} forces
the following equations
\begin{eqnarray}
\b_i+q_{\t(i),\t(i)+i-1}=0 \text{ for } i\neq k-\s+1 \label{eq1}, \\
\b_{k-\s+1}+q_{\s,k}=\zeta. \label{eq2}
\end{eqnarray}
By \eqref{knownbetas}, these are $k-1$ polynomial equations in $k-2$
variables, and the Jacobian at $0$ is the origin, so we have
finitely many solutions near the origin. Therefore, for some
$\zeta$, it follows that $T$ is not in $G^\s$.
\end{proof}

\bigskip

\noindent {\bf Case 3:} 
 $\s < k$ and $d=-1$ mod $\s$.

\begin{proof}
This case works very similarly to the previous one. Suppose $k-1>\s$,
that is, if $k=c\s-1$ where $c\ge 2$ (this holds because $k\ge \s$),
the condition is that $c\s-2 >\s$, which is true for all $k\geq 4$. 

Let $T$ be the transformation
\begin{equation}
T(e_i)=e_i \text{ for } i\neq k,k-1\ ;\ T(e_{k-1})=e_{k-1}+\zeta
e_\s \ ;\ T(e_{k})=e_{k}+\zeta e_\s
\end{equation}
First we check again that $T\in G_{\zdis_{\l^\s}}$. 
We have
\begin{eqnarray*}
\zdis_{\l^\s}[\s]=e_\s \ ;\\
\zdis_{\l^\s}[\s+1]=e_{\s+1}+e_{1}e_{\s} \ ;\\
\zdis_{\l^\s}[k]=e_{k}+\sum_{i=1}^{k-1}e_ie_{k-i} \ .\\
\end{eqnarray*}
An easy computation shows that
\begin{multline}
T\cdot \zdis_{\l^\s}=\zdis_{\l^\s}(e_1,\ldots, e_{k-2},e_{k-1}+\zeta
e_{\s},e_{k}+\zeta e_{\s+1})=\\
=\zdis_{\l^\s}[1] \wedge \ldots \wedge \zdis_{\l^\s}[k-2] \wedge
(\zdis_{\l^\s}[k-1]+\zeta \zdis_{\l^\s}[\s]) \wedge
(\zdis_{\l^\s}[k]+\zeta \zdis_{\l^\s}[\s+1] =\\
= \zdis_{\l^\s}[1] \wedge \ldots \wedge
\zdis_{\l^\s}[k]=\zdis_{\l^\s}.
\end{multline}

Now we prove that $T \not\in G^\s$ in a similar way to the second 
case above. 
Since $k-1 \neq -1$ mod $\s$ we can substitute $k-1$ instead
of $k$ in \eqref{maxdifference}:
\begin{equation}
n_{k-\s}=\l_{k-1}^\s-\l_\s^\s=\l_{k-\s}^\s-\l_1^\s.
\end{equation}
Moreover, we also get the extra equation
\begin{equation}
n_{k-\s}=\l_{k}^\s-\l_{\s+1}^\s,
\end{equation}
and similarly to \eqref{sigmadentry} and \eqref{firstrow} it follows
that
\begin{eqnarray}
(G^\s)_{\s,k-1}=\b_1^{\s-1}\b_{k-\s}+q_{\s,k-1}(\b_1,\ldots, \b_k);\\
(G^\s)_{\s+1,k}=\b_1^{\s}\b_{k-\s}+q_{\s+1,k}(\b_1,\ldots, \b_k);\\
(G^\s)_{1,k-\s}=\b_{k-\s}.
\end{eqnarray}
Since $T$ differs from the identity matrix only by the entries
\[(T)_{\s,k-1}=(T)_{\s+1,k}=\zeta,\]
the equality
\[T=G^\s(\b_1,\ldots ,\b_k)\]
forces $\b_{k-\s}=0,\b_1=1$ and the analogue of \eqref{eq1}
,\eqref{eq2}:
\begin{eqnarray}
\b_i+q_{\t(i),\t(i)+i-1}=0 \text{ for } i\neq k-\s \\
\b_{k-\s}+q_{\s,k-1}=\zeta  \\
\b_{k-\s}+q_{\s+1,k}=\zeta
\end{eqnarray}
which are, again, $k+1$ nondegenerate polynomial equations in $k-1$
variables, such that for some $\zeta$ there is no solution.
\end{proof}

We have now proved Lemma \ref{lempart2}, which together with Lemma \ref{lempart1} completes the proof
of  Theorem \ref{mainthm}.


\section{Geometric description of Demailly-Semple invariants}

As an immediate consequence of Corollary 6.3, we can now prove
Theorem \ref{Gros} in the case when $p=1$.

\begin{theorem} \label{Gros2} If $k \geq 2$ then
$\GG_{k}' = \UU_k$ is a Grosshans subgroup of the special linear group 
$SL(k)$, so that $\calo(\SL(k)^{\UU_{k}})^{\SL(k)}$
is a finitely generated complex algebra
 and moreover every linear action of $\UU_{k}$ or $\GG_k$ 
on an affine or projective variety $Y$ (with respect to an ample
linearisation) which extends to a linear action of $GL(k)$ has finitely generated invariants.
\end{theorem}

In particular we have the special case of Theorem \ref{fingenerated} when $p=1$.

\begin{theorem}\label{fingenerated1} 
The fibre $\calo((J_{k})_x)^{\UU_{k}}$ 
of the bundle $E_k^n$ is a finitely generated
graded complex algebra. 
\end{theorem}
\begin{proof}
We have
$$\calo((J_{k})_x)^{\UU_{k}} \cong (\calo((J_{k})_x) \otimes \calo(\SL(k)^{\UU_{k}})^{\SL(k)}
$$
which is finitely generated because $\calo(\SL(k)^{\UU_{k}})^{\SL(k)}$
is finitely generated and $\SL(k)$ is reductive.
\end{proof}

  Theorem \ref{mainthm} also allows us to describe the 
algebra $\calo(SL(k))^{\UU_{k}}$.
In \S 6 we constructed an embedding
of $SL(k) / \UU_k$ in the affine space $ \wsymk \otimes (\CC^k)^{\otimes K}$ for suitable large $K$,
and in Theorem \ref{mainthm} we proved that the boundary components of the 
closure $\overline{\SL(k)(p_k \otimes e_1^{\otimes K})}$ of its image  have codimension at least two.  Thus we obtain the following corollary of Theorem \ref{mainthm}:

 \begin{theorem} 
 (i) If $k \geq 4$ then
the canonical affine completion
$$\SL(k)/\!/\UU_k = \mathrm{Spec}(\calo(\SL(k))^{\UU_k})$$
of $\SL(k)/\UU_k$ is isomorphic to the closure $\overline{\SL(k)(p_k \otimes e_1^{\otimes K})}$ of the orbit ${\SL(k)(p_k \otimes e_1^{\otimes K})}
\cong \SL(k)/\UU_k$ of $p_k \otimes e_1^{\otimes K}$ in $ \wsymk \otimes (\CC^k)^{\otimes K}$ where $K = M(1+2+ \cdots + k)+1$ for any strictly
positive integer $M$;

(ii) The 
algebra 
$$\calo(SL(k))^{\UU_{k}}$$  is generated by the Pl\"{u}cker coordinates
on $\PP(\wsymk)$, which can be expressed as 
\[\{\Delta_{\bi_1,\ldots, \bi_s}:s\le k\},\]
where $\bi_j$ denotes a multi-index
corresponding to basis elements of $\Sym^{\leq k}(\CC^k)$, and  $\Delta_{\bi_1,\ldots, \bi_s}$ is the corresponding minor of $\phi(f'\ldots, f^{(k)}) \in \Hom(\CC^k, \Sym^{\leq k}(\CC^k))$, together with the coordinates
$f_1', \dots,f_k'$ of $f'$.
\end{theorem}

It follows immediately from this theorem that the non-reductive GIT quotient 
$$(J_k)_x /\!/\UU_k = \mathrm{Spec}(\calo((J_k)_x)^{\UU_k})$$
is isomorphic to the reductive GIT quotient 
$$((J_k)_x \times \overline{\SL(k)(p_k \otimes e_1^{\otimes K})})/\!/\SL(k).$$
This can be identified with the quotient of the open subset
$((J_k)_x \times \overline{\SL(k)(p_k \otimes e_1^{\otimes K})})^{ss}$
of $\SL(k)$-semistable points of 
$(J_k)_x \times \overline{\SL(k)(p_k \otimes e_1^{\otimes K})}$
by the equivalence relation $\sim$ such that $y \sim z$ if and only if the closures of the $\SL(k)$-orbits of $y$ and $z$ intersect in  
$((J_k)_x \times \overline{\SL(k)(p_k \otimes e_1^{\otimes K})})^{ss}$.
Equivalently it can be identified with the closed $\SL(k)$-orbits in 
$((J_k)_x \times \overline{\SL(k)(p_k \otimes e_1^{\otimes K})})^{ss}$.
Since 
$\overline{\SL(k)(p_k \otimes e_1^{\otimes K})}$ is the union of finitely
many $\SL(k)$-orbits, with stabilisers
$H_1 = \UU_k, H_2, \ldots, H_s$, say, we can stratify $(J_k)_x /\!/\UU_k$
so that the stratum corresponding to $H_j$ is identified with the
$H_j$-orbits in $(J_k)_x$ such that the corresponding $\SL(k)$-orbit in 
$(J_k)_x \times \overline{\SL(k)(p_k \otimes e_1^{\otimes K})}$
is semistable and closed in  
$((J_k)_x \times \overline{\SL(k)(p_k \otimes e_1^{\otimes K})})^{ss}$.

\begin{exit} When $k=2$ we have
\[J_2^{\mathrm{reg}}(1,2)=\{(f_1',f_2',f_1'',f_2'')\in
(\CC^2)^2;(f_1',f_2')\neq (0,0)\},\] and fixing a basis $\{e_1, e_2\}$
of $\CC^2$ and the induced basis $\{e_1,e_2,e_1^2,e_1e_2,e_2^2\}$ of $\CC^2 \oplus
\Sym^2\CC^2$, the map $\phi:J_2(1,2) = \Hom(\CC^2,\CC^2) \to \Hom(\CC^2,\Sym^{\le 2} \CC^2)$ of \eqref{homs} is given by
\[(f_1',f_2',f_1'',f_2'') \mapsto \left(
\begin{array}{ccccc}
f_1' & f_2' & 0 & 0 & 0 \\
\frac{1}{2!}f_1'' & \frac{1}{2!}f_2'' & (f_1')^2 & f_1'f_2' &
(f_2')^2
\end{array}
\right).
\]
The $2 \times 2$ minors of this $2 \times 5$ matrix  are $(f_1')^3,\,\, (f_1')^2 f_2',\,\,f_1'(f_2')^2,(f_2')^3$ and 
\[\Delta_{[1,2]}=f_1'f_2''-f_1''f_2'.\]
On $SL(2)$ we have $\Delta_{[1,2]}=1$ and  the algebra of invariants $\calo(SL(2))^{\UU_2}$  is generated by $f_1'$ and $f_2'$, as expected since
$\SL(2)/\UU_2 \cong \CC^2 \setminus \{ 0 \}$ and its canonical affine completion 
$\SL(2)/\!/\UU_2$ is $\CC^2$.

\end{exit}

\begin{exit} \label{7.3}  When $k=3$ the finite generation of  the Demailly-Semple algebra $\calo((J_k)_x)^{\UU_k}$ was proved by Rousseau in \cite{rousseau}.
We have
\[J_3^{\mathrm{reg}}(1,3)=\{(f_1',f_2',f_3',f_1'',f_2'',f_3'',f_1''',f_2''',f_3''')\in
(\CC^3)^3;(f_1',f_2',f_3')\neq (0,0,0)\},\] and if we fix a basis $\{e_1,e_2,e_3\}$ of
$\CC^3$ and the induced basis
\[ \{e_1,e_2,e_3,e_1^2,e_1e_2,e_2^2,e_1e_3,e_2e_3,e_3^2,e_1^3,e_1^2e_2,\ldots, e_3^3\}\]
of $\CC^3 \oplus \Sym^2 \CC^3 \oplus \Sym^3 \CC^3$, the map $\phi:\Hom(\CC^3,\CC^3)\to \Hom(\CC^3, \Sym^{\le 3}\CC^3)$ in \eqref{homs} sends 
\[(f_1',f_2',f_3',f_1'',f_2'',f_3'',f_1''',f_2''',f_3''')\] to a $3 \times 19$ matrix, whose first $9$ columns (corresponding to $\Sym^{\le 2}\CC^3$) are
\[
\left(
\begin{array}{ccccccccc}
f_1' & f_2' & f_3' & 0 & 0 & 0 & 0 & 0 & 0\\
\frac{1}{2!}f_1'' & \frac{1}{2!}f_2'' & \frac{1}{2!}f_3'' & (f_1')^2 & f_1'f_2' & (f_2')^2 & f_1'f_3' & f_2'f_3' & (f_3')^2 \\
\frac{1}{3!}f_1''' & \frac{1}{3!}f_2''' & \frac{1}{3!}f_3''' & f_1'f_1'' &
f_1'f_2''+f_1''f_2' & f_2'f_2'' & f_1'f_3''+f_3'f_1'' & f_2'f_3''+f_2''f_3' &
f_3'f_3'' 
\end{array}
\right)
,\]
and the remaining $10$ columns (corresponding to $\Sym^3\CC^3$) are
\[
\left(
\begin{array}{cccccccccc}
0 & 0 & 0 & 0 & 0 & 0 & 0 & 0 & 0 & 0\\
0 & 0 & 0 & 0 & 0 & 0& 0 & 0 & 0 & 0\\
(f_1')^3 & (f_1')^2f_2' & f_1'(f_2')^2 & (f_2')^3 & f_1'(f_3')^2 & (f_1')^2f_3' & (f_2')^2f_3' & f_2'(f_3')^2 & (f_3')^3 & f_1'f_2'f_3'
\end{array}
\right)
.\]
The $3 \times 3$ minors of this matrix together with $f_1',f_2',f_3'$
generate the algebra of invariants $\calo(SL(3))^{\UU_3}$.
\end{exit}

\section{Generalized Demailly-Semple jet bundles}

The aim of this section is to extend the earlier
constructions for $p=1$ to generalized Demailly-Semple invariant jet differentials
when $p>1$.

Let $X$ be a compact, complex manifold of dimension $n$. We fix a
parameter $1\le p \le n$, and study the maps $\CC^p \to X$.
Recall that as before we fix the degree $k$ of the map, and introduce the bundle
$J_{k,p} \to X$ of $k$-jets of maps $\CC^p \to X$, so that the
fibre over $x\in X$ is the set of equivalence classes of germs of holomorphic
maps $f:(\CC^p,0) \to (X,x)$, with the equivalence relation $f\sim
g$ if and only if all derivatives $f^{(j)}(0)=g^{(j)}(0)$ are equal for
$0\le j \le k$.
Recall also that $\GG_{k,p}$ is the group of $k$-jets of germs of biholomorphisms
of $(\CC^p,0)$, which has a natural fibrewise right action on $J_{k,p}$ with the matrix representation
given by
\begin{equation}
G_{k,p}=\left(
\begin{array}{ccccc}
\Phi_1 & \Phi_2 & \Phi_3 & \ldots & \Phi_k \\
0 & \Phi_1^2 & \Phi_1\Phi_2 & \ldots &       \\
0 & 0 & \Phi_1^3 & \ldots & \\
. & . & . & . & . \\
& & & & \Phi_1^k
\end{array}
\right),
\end{equation}
for $G_{k,p} \in \GG_{p,k}$
where
$\Phi_i \in \Hom(\Sym^i\CC^p,\CC^p)$ and $\det \Phi_1 \neq 0$. 
Recall also that $\GG_{k,p}$ is generated along its first $p$ rows, in the sense that the parameters in the first $p$ rows are independent, and
all the remaining entries are polynomials in these parameters. 
The parameters in the $(1,m)$ block are indexed by a basis of $\Sym^m(\CC^p) \times \CC^p$, so they are of the form $\a_{\nu}^l$ where $\nu \in {p+m-1 \choose m-1}$ is an $m$-tuple and $1\le l \le p$, and the polynomial in the $(l,m)$ block and entry indexed by $\tau=(\tau[1],\ldots, \tau[l]) \in {p+l-1 \choose l-1}$ and $\nu \in {p+m-1 \choose m-1}$ is given by
\begin{equation}
(G_{k,p})_{\tau,\nu}=\sum_{\nu_1+\ldots+\nu_l=\nu}\a_{\nu_1}^{\tau[1]}\a_{\nu_2}^{\tau[2]} \ldots \a_{\nu_l}^{\tau[l]}.
\end{equation}
Recall also  that $\GG_{k,p}=\UU_{k,p}\rtimes \GL(p)$ is an extension of its unipotent radical $\UU_{k,p}$ by
$\GL(p)$, and that the generalized Demailly-Semple jet bundle $E_{k,p,m} \to X$ of invariant jet differentials of
order $k$ and weighted degree $(m,\ldots, m)$ consists of the jet differentials which transform under any reparametrization $\phi\in \GG_{k,p}$ of
$(\CC^p,0)$ as
\[Q(f \circ \phi)=(J_{\phi})^mQ(f)\circ \phi,\]
where $J_\phi = \det \Phi_1$ denotes the Jacobian of $\phi$, so that $E_{k,p}=\oplus_{m \geq 0}
E_{k,p,m}$ is the graded algebra of $\GG_{k,p}'$-invariants
where $\GG_{k,p}' = \UU_{k,p} \rtimes \SL(p)$.



\subsection{Geometric description for $p>1$}
As in the case when $p=1$ our goal is to 
prove that $\GG_{k,p}'$ is a Grosshans subgroup of 
$\SL(\syms^{\leq k}(p))$ where $\syms^{\leq k}(p) = \sum_{i=1}^k \dim \Sym^i \CC^p$
 by finding a suitable embedding  of the quotient
$\SL(\syms^{\leq k}(p))/\GG_{k,p}'$. 

\begin{rem} In \cite{PR} Pacienza and Rousseau generalize the
inductive process given in \cite{dem} of constructing a smooth compactification of the
Demailly-Semple jet bundles. Using the concept of a directed
manifold, they define a bundle $X_{k,p} \to X$ with smooth fibres,
and the effective locus $Z_{k,p}\subset X_{k,p}$, and a holomorphic
embedding $J_{k,p}^{reg}/\GG_{k,p} \hookrightarrow Z_{k,p}$ which
identifies $J_{k,p}^{reg}/\GG_{k,p}$ with
$Z_{k,p}^{reg}=X_{k,p}^{reg}\cap Z_{k,p}$, so that
$Z_{k,p}$ 
 is a relative compactification of 
$J_{k,p}/\GG_{k,p}.$
We choose a different approach, generalizing the test curve model,
resulting in a holomorphic embedding of $J_{k,p}/\GG_{k,p}$ into a partial
flag manifold and a different compactification, which is a
singular subvariety of the partial flag manifold, such that the invariant
jet differentials of degree divisible by $\syms^{\leq k} (p)$ are given by
polynomial  expressions in the Pl\"{u}cker
coordinates.
\end{rem}

Fix $x \in X$ and an identification of $T_xX$ with $\CC^n$; then let
$J_k(p,n) = J_{k,p,x}$ as defined in {\S}2.  Let
\[\jetreg pn =\left\{\g \in J_k(p,n): \Gamma_1  
 \text{ is non-degenerate}
\right\}\]
where $\gamma$ is represented by
$$ \bu \mapsto \gamma(\bu)=\Gamma_1\bu+\Gamma_2\bu^2+\ldots +\Gamma_k\bu^k
$$  
with $\Gamma_i \in \Hom(\Sym^i\CC^p,\CC^p)$.
Let $N \ge n$ be any integer and define
\[\U_{k,p}=\left\{\Psi\in J_k(n,N):\exists \g \in \jetreg pn: \Psi \circ \g=0
\right\}.\]

\begin{rem}
The global singularity theory description of $\U_{k,p}$ is
\[\U_{k,p}\doteq \left\{p=(p_1,\ldots, p_N) \in J_k(n,N): \CC[z_1,\ldots ,z_n]/\langle p_1,\ldots ,p_N
\rangle \cong \CC[x,y]/\langle z_1, \ldots ,z_n \rangle^{k+1}
\right\}.\]
\end{rem}

Note, again, as in the $p=1$ case, that if $\g \in \jetreg pn$ is a
test surface of $\Psi \in \U_{k,p}$, and $\vp \in \GG_k$ is
a holomorphic reparametrization of $\CC^p$, then $\g \circ \vp$ is,
again, a test surface of $\Psi$:
\begin{diagram}[LaTeXeqno,labelstyle=\textstyle]
\label{basicideatwo}
  \CC^p & \rTo^\vp & \CC^p & \rTo^\g & \CC^n & \rTo^{\Psi} & \CC^N
\end{diagram}
\[\Psi \circ \g=0\ \ \Rightarrow \ \ \ \Psi \circ (\g \circ \vp)=0\]

\begin{exit} Let $k=2, p=2$ and let $\Psi(\bz)=\Psi'\bz+\Psi''\bz^2$ for $\bz \in \CC^n$, and
\[\g(u_1,u_2)=\g_{10}u_1+\g_{01}u_2+\g_{20}u_1^2+\g_{11}u_1u_2+\g_{02}u_2^2,\ \g_{ij}\in \CC^n.\]
Then $\Psi \circ \g=0$ has the form
\begin{eqnarray}\label{compp} & \Psi'(\g_{10})=0\ ;\
\Psi'(\g_{01})=0  \\ \nonumber &
\Psi'(\g_{20})+\Psi''(\g_{10},\g_{10})=0,\ ;\
\Psi'(\g_{11})+2\Psi''(\g_{10},\g_{01})=0,\ ;\
\Psi'(\g_{01})+\Psi''(\g_{01},\g_{01})=0, \\ \nonumber
\end{eqnarray}
\end{exit}

We introduce
\[
\cals_\g=\left\{\Psi \in J_k(n,N):\Psi \circ \g=0  \right\}
\]
and the following analogue of $\jetko 1n$:
\[\jetko nN=\left\{\Psi \in \jetk nN: \dim \ker \Psi=p
\right\}.\]
The proof of the following proposition is analogous to that of
Proposition 4.7 in \cite{bsz}, and we omit the details. We use the
notation
\[\syms^i(p)=\dim (\sym^i \CC^p) ;\ \syms^{\le k}(p)=\dim (\CC^p\oplus \sym^2 \CC^p \oplus \ldots \oplus \sym^k
\CC^p)=\sum_{i=1}^k \syms^ip.\]

\begin{prop}
\label{modelpropsgen}
\begin{enumerate}
\item[(i)] If $\g\in \jetreg pn$ then $\cals_\g \subset \jetk nN$ is a linear
  subspace of codimension $N\syms^{\le k}(p)$.
\item[(ii)] For any $\g \in \jetreg pn$, the subset $\cals_\g \cap \jetko nN$ of
$\cals_\g$ is dense.
\item[(iii)] If $\Psi \in \jetko nN$, then $\Psi$ belongs
  to at most one of the spaces $\cals_\g$. More precisely, if
$
\g_1,\g_2\in\jetreg pn,\;\ \Psi\in \jetko nN \text{ and }
  \Psi\circ \g_1=\Psi\circ \g_2=0,
$
then there exists
  $\vp \in \jetreg pp$ such that $\g_1=\g_2 \circ \vp$.
\item[(iv)] Given $\g_1,\g_2\in\jetreg 1n$, we have
  $\cals_{\g_1}=\cals_{\g_2}$ if and only if there is some
  $\vp \in \jetreg 11$ such that $\g_1=\g_2 \circ \vp$.
\end{enumerate}
  \end{prop}

With the notation
\[\U_{k,p}=\U_{k,p} \cap \jetko nN,\]
we deduce from Proposition \ref{modelpropsgen} the following
\begin{corollary}
$\U_{k,p}^0$ is a dense subset of $\U_{k,p}$, and $\U_{k,p}^0$ has a
fibration over the orbit space $\jetreg pn/\jetreg pp=\jetreg
pn/\GG_{k,p}$ with linear fibres.
\end{corollary}

\begin{rem}
In fact, Proposition \ref{modelpropsgen} says a bit more, namely that
$\U_{k,p}^0$ is fibrewise dense in $\U_{k,p}$ over $\jetreg
pn/\GG_{k,p}$, but we will not use this stronger statement.
\end{rem}

By the first part of Proposition \ref{modelpropsgen} the assignment
$\g \to \cals_\g$ defines a map
\[\n: \jetreg pn \rightarrow \grass(kN,\jetk nN)\]
which, by the fourth part, descends to the quotient
\begin{equation}\label{embedding1}
\bar{\n}:\jetreg pn/\GG_{k,p} \hookrightarrow \grass(kN,\jetk nN)
\end{equation}
 (cf. Proposition
\ref{propgrass}).
Next, we want to rewrite this embedding in terms of the
identifications introduced in  \S\ref{quotientflag}. So we
\begin{itemize}
\item identify  $\jetk pn\text{ with }\Hom(\CC^{\syms^1p}\oplus \ldots \oplus \CC^{\syms^kp},\CC^n)=
\Hom(\CC^{\syms^{\leq k}(p)},\CC^n)$
where $\syms^jp=\dim\sym^j \CC^p$ and $\syms^{\leq k}(p) = \sum_{j=1}^k \syms^jp$;
\item identify  $\jetk n1{}^* \text{ with }
\symdot=\oplus_{l=1}^k \Sym^l \CC^n$.
\end{itemize}
We think of an element $v$ of $\Hom(\CC^{\syms^{\le k}(p)},\CC^n)$ as an $n
\times \syms^{\le k}(p)$ matrix, with column vectors in $\CC^n$. These
columns correspond to basis elements of $\CC^{\syms^1p}\oplus
\ldots \oplus \CC^{\syms^kp}$, and the columns in the $i$th
component are indexed by $i$-tuples $1\le t_1 \le t_2 \le \ldots
\le t_i \le p$, or equivalently by
\[(e_{t_1}+e_{t_2}+\ldots +e_{t_i}) \in \ZZ_{\ge 0}^p\]
 where $e_j=(0,\ldots,1,\ldots ,0)$ with $1$
in the $j$th place,
giving us 
$$v = (v_{10,\ldots 0},v_{01\ldots 0},\ldots ,v_{0\ldots 0k}) \in 
\Hom(\CC^{\syms^{\le k}(p)},\CC^n). $$
 The elements of $\jetreg pn$ correspond to matrices whose first $p$ columns are linearly independent. When $n \ge \syms^{\le k}(p)$  there is a smaller dense open subset $\jetnondeg pn \subset \jetreg pn$ consisting of the $n
\times \syms^{\le k}(p)$ matrices of rank $\syms^{\le k}(p)$.

Define the following map, whose components correspond to the
equations in \eqref{compp}:
\begin{eqnarray}\label{embeddef}
& \phi: \Hom(\CC^{\syms^{\leq k}(p)},\CC^n) \to \Hom(\CC^{\syms^{\leq k}(p)}, \symdot) \\
\nonumber & (v_{10,\ldots 0},v_{01\ldots 0},\ldots ,v_{0\ldots 0k})
\mapsto (\ldots , \sum_{\bs_1+\bs_2+\ldots
+\bs_j=\bs}v_{\bs_1}v_{\bs_2}\ldots v_{\bs_j}, \ldots ),
\end{eqnarray}
where on the right hand side 
$\bs\in \ZZ_{\ge 0}^p$.

\begin{exit} If $k=p=2$ then $\phi$ is given by
\[\phi(v_{10},v_{01},v_{20},v_{11},v_{02})=(v_{10},v_{01},v_{20}+v_{10}^2,v_{11}+2v_{10}v_{01},v_{02}+v_{01}^2). \]
\end{exit}

Let $P_{k,p}\subset \GL_{\syms^{\leq k}(p)}$ denote the standard
parabolic subgroup
with Levi subgroup $$\GL(\syms^1p)\times \ldots \times \GL(\syms^k p),$$
where $\syms^jp=\dim\sym^j \CC^p$ and $\syms^{\leq k}(p) = \sum_{j=1}^k \syms^jp$.
Then \eqref{embedding1} has the following reformulation, analogous
to Proposition \ref{embedfinal}.

\begin{prop}
The map $\phi$ in \eqref{embeddef} is a $\GG_{k,p}$-invariant
algebraic morphism
\[ \phi: \jetreg pn \rightarrow \Hom(\CC^{\syms(p)},\symdot)\]
which induces an injective map $\phi^\grass$ on the
$\GG_{k,p}$-orbits:
\[\phi^\grass: \jetreg pn /\GG_{k,p} \hookrightarrow
\grass_{\syms^{\le k}(p)}(\symdot) \]
and 
\[\phi^\flag: \jetreg pn /\GG_{k,p} \hookrightarrow
\flag_{\syms^1(p),\ldots,
\syms^k(p)}(\symdot)\hookrightarrow
\Hom(\CC^{\syms(p)},\symdot)/P_{k,p} .\]
Composition with the Pl\"{u}cker embedding gives
\[\phi^\Proj=\mathrm{Pluck} \circ \phi^\grass: \jetreg pn /\GG_{k,p} \hookrightarrow \PP(\wedge^{\syms^{\le k}(p)}\symdot). \]
\end{prop}
As in the case when $p=1$, we introduce the following notation
\[X_{n,k,p}=\phi^{\Proj}(\jetreg pn),\,\,\, Y_{n,k,p}=\phi^{\Proj}(\jetnondeg pn) \subset \PP(\wedge^{\syms^{\le k}}(\symdott)).\]



\begin{definition}
Let $n\ge \syms^{\le k}(p)=\syms^1(p)+\ldots +\syms^k(p)$. Then the open subset of $\PP(\wedge^{\syms^{\le k}(p)}(\symdot))$ where the projection to $\wedge^{\syms^{\le k}(p)}\CC^n$ is nonzero is denoted by $A_{n,k,p}$. 
\end{definition}

Since $\phi^\grass$ and $\phi^\Proj$ are $\GL(n)$-equivariant, and for $n \ge \syms^{\le k}(p)$ the action of $\GL(n)$ is transitive on $\Hom^{\mathrm{nondeg}}(\CC^{\syms^{\le k}(p)},\CC^n)$, we have  

\begin{lemma}
\begin{enumerate}
\item[(i)] If $n \ge \syms^{\le k}(p)$ then $X_{n,k,p}$ is the $\GL(n)$ orbit of
\begin{equation}\label{zdisdefp}\zdis=\phi^{\Proj}(e_1,\ldots, e_{\sym^{\le k}(p)})=[\wedge_{j_1+\ldots +j_p\le k} \sum_{\bi_1+\ldots +\bi_s=(j_1,\ldots,j_p)}e_{\bi_1}\ldots e_{\bi_s} ]
\end{equation}
in $\PP(\wedge^{\syms^{\le k}(p)}(\symdot))$.
\item[(ii)] If $n \ge \syms^{\le k}(p)$ then $X_{n,k,p}$ and $Y_{n,k,p}$ are finite unions of $\GL(n)$ orbits.
\item[(iii)] For $k>n$ the images $X_{n,k,p}$ and $Y_{n,k,p}$ are $\GL(n)$-invariant quasi-projective varieties, though they have no dense $\GL(n)$ orbit.   
\end{enumerate}
\end{lemma}

Similar statements hold for the closure of the image in the Grassmannian
$${\grass}_{\syms^{\leq k}(p)}(\sym^{\leq k}\CC^n)$$ (or equivalently in the projective space $\PP(\wedge^{\syms^\le k(p)}(\symdot))$).

\begin{lemma} Let $n \ge \syms^{\le k}(\CC^n)$; then  
\begin{enumerate}
\item[(i)] $A_{n,k,p}$ is invariant under the $GL(n)$ action on $\PP(\wedge^{\syms^\le k(p)}(\symdot))$; 
\item[(ii)] $X_{n,k,p} \subset A_{n,k,p}$, although $Y_{n,k,p} \nsubseteq A_{n,k,p}$;
\item[(iii)] $\bX_{n,k,p}$ is the union of finitely many $GL(n)$-orbits. 
\end{enumerate}
\end{lemma}

\section{Affine embeddings of $\SL(\symkp)/\GG_{k,p}$}

In this section we study the case when $n=\symkp$ and so  $\GL(n) \subset \jetreg pn$. In the previous section we embedded $\jetreg pn/\GG_{k,p}$ in the affine space $A_{n,k,p} \subset \PP(\wedge^n \symk)$, which can be restricted to $\GL(n)$ to give us an embedding  
\[\GL(n)/\GG_{k,p} \hookrightarrow \PP(\wedge^n \symk)\]
as the $\GL(n)$ orbit of 
\[[\ldots \wedge \sum_{|\bs|=j}\sum_{\bs_1+\bs_2+\ldots
+\bs_j=\bs}e_{\bs_1}e_{\bs_2}\ldots e_{\bs_j} \wedge \ldots ].\]
Equivalently we have $\SL(n)/(\SL(n) \cap \GG_{k,p})=\SL(n)/\GG_{k,p}' \rtimes 
F_{k,p}$ embedded in $\wsymk$ as the $\SL(k)$ orbit of 
\[p_{k,p}=\ldots \wedge \sum_{|\bs|=j}\sum_{\bs_1+\bs_2+\ldots
+\bs_j=\bs}e_{\bs_1}e_{\bs_2}\ldots e_{\bs_j} \wedge \ldots ,\]
where $\SL(n) \cap \GG_{k,p}$ is the semi-direct product $\GG_{k,p}' \rtimes 
F_{k,p}$ of $\GG_{k,p}'$ by the finite group $F_{k,p}$ of $l_{k,p}$th roots of unity in $\CC$ for $l_{k,p}=\sum_{i=1}^k i \syms^i p$.
In analogy with $\S$6 we can consider an embedding of $\SL(n)/\GG_{k,p}'$ in
\[\wsymn \otimes (\wedge^p(\CC^n))^{\otimes K}\]
for suitable $K$ and its closure in this affine space.
We expect the following result generalising Theorem \ref{mainthm}.

\begin{conjecture}\label{mainthmp} Let $K=M(\sum_{i=1}^k i \syms^i p)+1$ where $M \in \mathbb{N}$. Then the point 
\[p_{k,p} \otimes (e_1\wedge \ldots \wedge e_p)^{\otimes K} \in 
\wsymn \otimes (\wedge^p(\CC^n))^{\otimes K} \]
where 
\[p_{k,p}=\ldots \wedge \sum_{|\bs|=j}\sum_{\bs_1+\bs_2+\ldots
+\bs_j=\bs}e_{\bs_1}e_{\bs_2}\ldots e_{\bs_j} \wedge \ldots\]
has stabiliser $\GG_{k,p}'$ in $\SL(n)$, and 
the closure of its $\SL(n)$ orbit
$$\overline{\SL(n)(p_{k,p} \otimes (e_1\wedge \ldots \wedge e_p)^{\otimes K})}$$ is the union of the orbit of $p_{k,p} \otimes (e_1\wedge \ldots \wedge e_p)^{\otimes K}$ and finitely many other $\SL(n)$-orbits, all of 
which have codimension at least two 
if $k$ is large enough (depending on $p$) and $M$  is sufficiently large
(depending on $k$ and $p$). 
\end{conjecture}
  
The proof of Conjecture \ref{mainthmp} should be similar to that of Theorem \ref{mainthm}, 
with the r\^{o}le of the Borel subgroup $B_k$ of $\SL(k)$ played by the standard
parabolic subgroup $P \subset \SL(n)$  which stabilises the filtration
\[0 \subset \CC^p = \CC e_1 \oplus \ldots \oplus \CC e_p \subset 
\CC^p \oplus \Sym^2\CC^p \subset \ldots \subset 
\CC^p \oplus \Sym^2\CC^p \oplus \cdots \oplus \Sym^k\CC^p
= \CC^n.\] 

It follows immediately from Conjecture \ref{mainthmp} that we would have

\begin{conjecture} \label{Gros3} If  $p\geq 1$ and
$k$ is large enough (depending on $p$) then the 
reparametrisation group
$\GG_{k,p}'$ is a subgroup of the special linear group 
$ \SL(\mathrm{sym}^{\le k}p)$, where 
$$ \mathrm{sym}^{\le k}p = \sum_{i=1}^k \dim \sym^i \CC^p =  \left( \begin{array}{c} k+p-1\\ k-1 \end{array} \right), $$ 
such that the algebra of invariants 
$$\calo(\SL(\mathrm{sym}^{\le k}p))^{\GG_{k,p}'}$$
is finitely generated,
so that every linear action of $\GG_{k,p}$ or $\GG_{k,p}'$ on an affine or
projective variety (with respect to an ample linearisation) which extends to a linear action of $ GL(\mathrm{sym}^{\le k}p)$ has finitely generated invariants.
\end{conjecture}

In particular we would have

\begin{conjecture}\label{fingenerated3} If  $p\geq 1$ and
$k$ is large enough (depending on $p$) then the fibres $\calo((J_{k,p})_x)^{\GG_{k,p}'}$
of the bundle $E_{k,p}^n$ are finitely generated
graded complex algebras.
\end{conjecture}

We would also obtain geometric descriptions of the associated affine varieties
 $$\mathrm{Spec} (\calo(\SL(\mathrm{sym}^{\le k}p))^{\GG_{k,p}'}) $$  and $\mathrm{Spec} (\calo((J_{k,p})_x)^{\GG_{k,p}'}) $  generalising those in $\S$7.

\end{document}